\documentclass{amsart}

\usepackage{geometry}
\usepackage{amsfonts}
\usepackage{amssymb}
\usepackage{amsmath}
\usepackage{url}
\usepackage{hyperref}
\usepackage{xy}
\input xy
\xyoption{all}

\usepackage{cleveref}
\usepackage{amsthm}
\usepackage{tikz}

\newcommand{\spec}{\mathrm{Spec}}
\newcommand{\Tmf}{\mathrm{Tmf}}
\newcommand{\A}{\mathbb{A}}

\newcommand{\tmf}{\mathrm{tmf}}
\newcommand{\gpd}{\mathbf{Gpd}}
\renewcommand{\A}{\mathcal{A}_*}

\newcommand{\st}{\mathrm{Stack}}
\newcommand{\otop}{\mathcal{O}^{\mathrm{top}}}
\newcommand{\mell}{M_{\overline{ell}}}
\renewcommand{\hom}{\mathrm{Hom}}

\renewcommand{\rightrightarrows}{\begin{smallmatrix} \to \\
\to \end{smallmatrix} }
\newcommand{\triplearrows}{\begin{smallmatrix} \to \\ \to \\ 
\to \end{smallmatrix} }

\theoremstyle{theorem}
\newtheorem{theorem}{Theorem}[section]
\newtheorem{lemma}[theorem]{Lemma}
\newtheorem{corollary}[theorem]{Corollary}
\newtheorem{proposition}[theorem]{Proposition}

\newtheorem*{tthm}{Technical theorem}
\newtheorem*{lm}{Lemma}

\theoremstyle{definition}
\newtheorem{definition}[theorem]{Definition}
\newtheorem{example}[theorem]{Example}
\newtheorem*{conv}{Convention}
\newtheorem*{cons}{Construction}
\newtheorem*{remark}{Remark}

\begin{document}

\theoremstyle{definition}
\title{The homology of $\tmf$}
\date{\today}
\author{Akhil Mathew}
\email{amathew@math.harvard.edu}
\address{Department of Mathematics, Harvard University, Cambridge, MA}

\begin{abstract}
We compute the mod $2$ homology of the spectrum $\tmf$ of topological modular
forms by proving a 2-local equivalence $\tmf \wedge DA(1) \simeq \tmf_1(3)
\simeq BP\left
\langle 2\right\rangle$, where $DA(1)$ is an eight cell complex whose
cohomology  doubles  the subalgebra $\mathcal{A}(1)$ of the Steenrod
algebra generated by $\mathrm{Sq}^1$ and $\mathrm{Sq}^2$. To do so, we give,
using the
language of stacks, a modular
description of the elliptic homology of $DA(1)$ via level three structures. We
briefly discuss analogs at odd primes and recover the stack-theoretic
description of the Adams-Novikov spectral
sequence for $\tmf$. 
\end{abstract}
\setcounter{tocdepth}{2}

\maketitle
\section{Introduction}

\subsection{Results}
Let $\tmf$ be the $E_\infty$-ring spectrum 
of connective \emph{topological modular forms} of Goerss, Hopkins, Mahowald,
Miller, and Lurie. See \cite{goersssurvey} for a survey and
\cite{rezk512, TMF} for detailed
treatments. 
The spectrum $\tmf$ is constructed from a derived version of the moduli  stack $\mell$
of elliptic curves, and its homotopy groups approximate both the
stable homotopy groups of spheres and the ring of integral modular forms. 

The primary goal of this paper is to compute the mod 2 cohomology of $\tmf$, as a
module over the Steenrod algebra.  Namely,
we give a proof of the following result: 

\begin{theorem}[Hopkins-Mahowald \cite{HM}] 
There is an isomorphism
\begin{equation} \label{Ctmf} H^*(\tmf; \mathbb{Z}/2) \simeq \mathcal{A} // \mathcal{A}(2)
\stackrel{\mathrm{def}}{=} 
\mathcal{A} \otimes_{\mathcal{A}(2)} \mathbb{Z}/2,  \end{equation}
where $\mathcal{A}$ is the mod 2 Steenrod algebra and $\mathcal{A}(2) \subset \mathcal{A}$ is the subalgebra generated by
$\mathrm{Sq}^1, \mathrm{Sq}^2, \mathrm{Sq}^4$.
\end{theorem} 
 The computation is carried out
by exhibiting a 2-local eight cell complex $DA(1)$ and 
demonstrating an equivalence (also due to \cite{HM}),
\begin{equation} \label{tmfDA1} \tmf \wedge DA(1) \simeq BP \left \langle
2\right\rangle,  \end{equation}
which is a $\tmf$-analog of Wood's theorem $ko \wedge
\Sigma^{-2} \mathbb{CP}^2 \simeq ku$. This equivalence implies the result on
$H^*(\tmf; \mathbb{Z}/2)$ using Hopf algebra manipulations, and it enables the
description of the Adams-Novikov spectral sequence of $\tmf$ via the
Weierstrass Hopf algebroid or the moduli stack of cubic curves.

The form of $BP\left \langle 2\right\rangle$ encountered
can be identified with the spectrum $\tmf_1(3)$ of  connective topological modular forms
\emph{of level 3} studied by Hill-Lawson \cite{HillLawson} and by
Lawson-Naumann \cite{LN}. 
In particular, we prove: 

\begin{theorem} At the prime $2$, we have an equivalence of $\tmf$-modules
\begin{equation} \label{tmfDA12} \tmf \wedge DA(1) \simeq \tmf_1(3).
\end{equation}
\end{theorem} 

At the prime 3, there is an analog as well: one has a three cell complex $X_3$ such that
\begin{equation} \label{tmfDA13} \tmf \wedge X_3 \simeq \tmf_1(2) .
\end{equation}

\subsection{Methods} We now give a brief overview of the method of proof used
in this paper. 

Suppose we knew the homology of $\tmf$, i.e., \eqref{Ctmf}. 
In this case, one could construct the eight cell complex $DA(1)$ such that, as
an $\mathcal{A}(2)$-module, we have $$H^*(DA(1); \mathbb{Z}/2)  \simeq
\mathcal{A}(2)/(\mathrm{Sq}^1
)$$ 
We would then see by direct computation that
$\tmf \wedge DA(1)$ had the same homology as $BP\left \langle 2\right\rangle$
and could produce the equivalence 
\eqref{tmfDA1} with some more effort (or by appealing to \cite{AL15} at least
after 2-completion).  
In this paper, we will work in reverse: we will prove \eqref{tmfDA1} directly
by working with stacks and then deduce the description of the homology. 

Our proof of \eqref{tmfDA1}  proceeds by
first replacing $\tmf$ by $\Tmf$, 
the non-connective, non-periodic version of $\tmf$, and by working with the
derived version of $M_{\overline{{ell}}}$. 
In particular, we calculate $\Tmf_*(DA(1))$. To do so, we describe the module $E_0(DA(1))$, for $E$ an elliptic homology theory, in
terms of $\mell$: in other words, we identify a certain
vector bundle on the moduli stack of elliptic curves as arising from an
eight-fold cover of $(\mell)_{(2)}$. 

\begin{tthm}
Let $E$ be an elliptic homology theory associated to a generalized elliptic curve $C \to
\spec (R)$, classified by a flat map $\spec(R) \to \mell$.
Then there is a natural identification
of $R = E_0$-modules between 
$E_0( DA(1))$ 
and the  universal $R$-algebra over which $C$ acquires a
$\Gamma_1(3)$-structure.  
\end{tthm}

This proof of identification relies on a basic technical trick, of which we explain a more elementary
form. 
\begin{cons}
Consider the scheme $\mathbb{P}^n_{\mathbb{C}} = \left( \spec
(\mathbb{C}[x_0,
x_1, \dots, x_n]) \setminus (0, 0, \dots, 0) \right) / \mathbb{G}_m$. 
This scheme is very well-behaved; it is smooth and proper. It is an open
substack of the Artin stack $\mathfrak{X} \subset 
\spec (\mathbb{C}[x_0,
x_1, \dots, x_n])/\mathbb{G}_m$. 
\end{cons}

On the one hand, the geometry of $\mathfrak{X}$
is much worse than that of $\mathbb{P}^n_{\mathbb{C}}$: it has a special point $x$, given by the image of the origin, whose
stabilizer is a $\mathbb{G}_m$. On the other hand, the study of vector bundles
or coherent sheaves on $\mathfrak{X}$ is vastly simpler than the analogous
study on $\mathbb{P}^n_{\mathbb{C}}$, as they are given by
graded modules over a polynomial ring. 
As an example, one has an analog of Nakayama's lemma:

\begin{lm}
Any morphism of coherent sheaves
\( \mathcal{F}  \to \mathcal{G} \)
on $\mathfrak{X}$ inducing a surjection at the point $x$ is a surjection
globally. 
\end{lm}

In this paper, we will work with the moduli stack $\mell$ of
generalized elliptic curves, which is analogous to
$\mathbb{P}^n_{\mathbb{C}}$: it has good geometry, but it generally hard to
study vector bundles on it. There is, however, another moduli stack $M_{cub}$
classifying \emph{cubic} curves,
which contains $\mell$ as an open substack. As in the above analogy,
$M_{cub}$ has a distinguished point with a bigger stabilizer for which one can
prove an analog of Nakayama's lemma.
Our main identification relies on the observation that the
association $E \mapsto E_0(DA(1))$, which defines a vector  bundle
on $(\mell)_{(2)}$, actually canonically extends to the larger stack
$(M_{cub})_{(2)}$. 
As a result, it is possible to get a handle on this vector bundle using the
distinguished point in $M_{cub}$. 
After identifying the vector bundle, we use the descent spectral sequence
to compute $\Tmf_*(DA(1))$.
Finally, the description of $\tmf_*(DA(1))$
from that of $\Tmf_*(DA(1))$ follows from the gap theorem in $\pi_* \Tmf$. 

The above summarizes the key technical work in the paper. 
In the rest of the paper, we calculate the homology $H_*( \tmf; \mathbb{Z}/2)$
by studying the map $\tmf \to \tmf \wedge DA(1) \simeq BP\left \langle
2\right\rangle$ and determining the image in homology, as the homology of
$BP\left \langle 2\right\rangle$ is known as a subalgebra of the dual Steenrod
algebra. 
The precise determination of the image relies on a few techniques with Hopf
algebras.

This paper is organized as follows. Section 2 reviews the language of stacks
and, in
particular, the role of the moduli stack $M_{FG}$ of formal groups and states
the results we need about $\tmf$.  Section 3 is purely algebraic and describes a vector
bundle on the stack $\mell$. Section 4 shows that this vector  bundle arises
from the eight cell
complex $DA(1)$ and discusses the analog at odd primes. Section 5 contains the
main remaining computations, in particular, of the homology.

\subsection{Previous work}
There is a significant literature on $\tmf$ and the homology of $\tmf$ has
certainly been treated before. 
For instance, the notes of Rezk \cite[sec. 20--21]{rezk512} give an entirely different
approach to the calculation of $H_*(\tmf; \mathbb{Z}/2)$. Rezk's starting point
is different from ours; we take the description via sheaves of spectra as
given, whereas Rezk assumes the $\tmf$-homology of the Thom spectrum $X(4)$. 
In particular, the calculation of $MU_* \tmf$  is carried out in 
\cite[Prop. 20.4]{rezk512} and we will reprove it here as well. 

The equivalence \eqref{tmfDA1} is due to Hopkins-Mahowald. The existence of an
eight cell complex 
$DA(1)$ with the property that $\tmf \wedge DA(1)$ is complex-orientable arises
from the heuristic that the moduli stack of cubic curves has an eight-fold 
cover at the prime 2 which is the quotient by $\mathbb{G}_m$ of an affine
scheme (cf. the notes to Hopkins's talk
in \cite[Ch. 9]{TMF}).

\subsection*{Acknowledgments}
I would like to thank heartily Mike Hopkins for 
suggesting this project and sharing his ideas. I would also like to thank
Mark Behrens, Dustin
Clausen, Tyler Lawson, Jacob Lurie, Lennart Meier, Niko Naumann, and Vesna
Stojanoska for helpful discussions, and the referee for numerous comments. 
\section{The language of stacks}

Let $X$ be a spectrum. The homotopy groups $\pi_* X$ may be complicated, but
often their calculation can be attacked by choosing an appropriate resolution
of $X$ by simpler spectra. This formalism can be expressed efficiently using
the language of stacks. 
The language  has been described in the course notes \cite{COCTALOS, chromatic} and
in the talk of Hopkins \cite[Ch. 7]{TMF}. Other references on the this viewpoint, especially on
the moduli stack
of formal groups,  are \cite{naumann} and
\cite{goerss}. We will briefly summarize what we need below.

\subsection{Stacks and spectra}

Suppose that $R$ is a fixed  $A_\infty$-ring
spectrum. The structure on $R$ enables one to build the cobar
construction, a cosimplicial spectrum $\mathrm{CB}^\bullet(R)$
with 
$\mathrm{CB}^s(R) = R^{\wedge (s+1)}$ and with the coface and codegeneracy
maps arising in a standard manner from the unit $S^0 \to R$ and the multiplication $R \wedge R \to R$. 
Under good conditions,  for a spectrum $X$, the cosimplicial diagram
\begin{equation}  R^\bullet(X) \stackrel{\mathrm{def}}{=} X \wedge
\mathrm{CB}^\bullet(R) =  \{ X \wedge R
\rightrightarrows X \wedge R \wedge R  \triplearrows \dots  \}  \end{equation}
will be a resolution of $X$ in the sense that the natural map
$X \to \mathrm{Tot}( R^\bullet(X))$ 
is an equivalence.
If so, then one has a homotopy spectral sequence
\begin{equation} \label{ss} E_2^{s,t} = \pi^s\pi_{t} R^\bullet(X) \implies
\pi_{t-s} X.  \end{equation}
If $R = MU$ is complex bordism, and $X$ is connective, then  the associated spectral
sequence is the Adams-Novikov spectral sequence.

Suppose now that $R$ is also homotopy commutative, and for each $s$, $\pi_*
R^{\wedge (s+1)}$ is concentrated in even degrees. 
Then we get a cosimplicial commutative ring $\pi_* R^{\wedge (s+1)}$, over which $\pi_*
(R^{\wedge s+1} \wedge X)$ is a cosimplicial module. 
If $R_* R$ is flat over
$R_*$,  the diagram
\[ \pi_* R \rightrightarrows \pi_* ( R \wedge R ) \triplearrows \dots  \]
is determined by its 2-truncation, and it is a commutative \emph{Hopf
algebroid}. This presents  a \emph{stack} $\mathfrak{X}$. 
Furthermore, the cosimplicial $\pi_*
R^\bullet$-module $\pi_* R^\bullet(X)$ defines a quasi-coherent sheaf
 on 
$\mathfrak{X}$, i.e., a
comodule over the Hopf algebroid $(R_*, R_*R)$. 
We shall denote this by $\mathcal{F}(X)$.

The chain complex
\[ \pi_* (R \wedge X) \to \pi_* ( R \wedge R \wedge X) \to \dots , \]
whose cohomology is the $E_2$ page of \eqref{ss}, can be identified with the
 the \emph{cobar complex}
\[ R_*(X) \to R_*(R) \otimes_{R_*} R_* X \to \dots,  \]
which computes the cohomology of the sheaf $\mathcal{F}(X)$. In particular, the spectral sequence \eqref{ss} 
can be written as 
\[ H^s( \mathfrak{X}, \mathcal{F}(X)) \implies \pi_{t-s} X.  \]

It is convenient to make one further modification. The rings $R_*, R_* R$ are graded rings, and the 
sheaves $\mathcal{F}(X)$ come from graded comodules over the Hopf
algebroid $(R_*, R_* R)$. Let us now suppose that 
$R_*, R_* R$ are \emph{evenly graded.} The grading determines (and is
equivalent to) a $\mathbb{G}_m$-action on the Hopf algebroid $(R_*, R_* R)$, or
on the stack $\mathfrak{Y}$,
such that an element in degree $2k$ is acted on by $\mathbb{G}_m$ with
eigenvalue given by 
the character $\chi_k\colon \mathbb{G}_m\to \mathbb{G}_m$, $\chi_k(u) = u^k$. 
We can regard
$\mathcal{F}(X)$ as the sum of comodules $\mathcal{F}_{\mathrm{even}}(X) \oplus
\mathcal{F}_{\mathrm{odd}}(X)$, where each of the two summands inherits a
$\mathbb{G}_m$-action in a similar manner. That is,
$\mathcal{F}_{\mathrm{even}}(X)$ is given a $\mathbb{G}_m$-action in the same
manner, and $\mathbb{G}_m$ acts on $\mathcal{F}_{2k+1}(X)$ by the character
$\chi_k$. 

If we form the stack $\mathfrak{Y} = \mathfrak{X}/\mathbb{G}_m$, then $\mathfrak{Y}$
comes with a tautological line bundle $\omega$, corresponding to the $(R_*, R_*R)$-comodule  $R_{\ast + 2}$.
Moreover,
$\mathcal{F}_{\mathrm{even}}$ and $\mathcal{F}_{\mathrm{odd}}$ descend to
functors into $\mathrm{Mod}(\mathfrak{Y})$, which is equivalent to 
the category of \emph{evenly graded} comodules over $(R_*, R_* R)$.
The spectral sequence can be written
\begin{equation} \label{ANSSstack}  E_2^{s,t} \implies \pi_{t-s} X, \quad \quad E_2^{s, t} =  \begin{cases} 
H^{s}( \mathfrak{Y}, 
\mathcal{F}_{\mathrm{even}}(X) \otimes \omega^{t'}) & \text{if } t = 2t'  \\
H^{s}( \mathfrak{Y}, 
\mathcal{F}_{\mathrm{odd}}(X) \otimes \omega^{t'})  & \text{if } t  = 2t' + 1
 \end{cases}. \end{equation}

We now recall Quillen's theorem.
\begin{definition} 
$M_{FG}$ is the \emph{moduli stack of formal groups.} In other words, 
$M_{FG}$ is the 2-functor
\( M_{FG} \colon  \mathrm{Ring}\to \gpd  \)
assigning to any commutative ring $A$ the groupoid of formal group schemes
$X \to \spec (A)$ which are Zariski locally (on $A$) isomorphic to
$\mathrm{Spf} (A[[x]])$
as pointed formal schemes. 
\end{definition} 

When $R = MU$ (as we will henceforth assume),
Quillen's theorem states that there is an equivalence   
between the stack that
one builds from the evenly graded Hopf algebroid $(MU_*, MU_* MU)$ and the
moduli stack of formal groups.

We
can often describe the quasi-coherent sheaves $\mathcal{F}_{\mathrm{even}}(X),
\mathcal{F}_{\mathrm{odd}}(X)$ on $M_{FG}$ in
terms of the geometry of formal groups. 

\begin{example}
The line bundle $\omega$ described above on $M_{FG}$ (that is, associated to
the $(MU_*, MU_*MU)$-comodule $MU_{* + 2}$, which arises topologically from
$S^{-2}$) assigns to a formal group $X$ the \emph{cotangent space}
$\mathcal{O}_{X}(-e)/\mathcal{O}_X(-2e)$ of functions on $X$ that vanish at
zero, modulo functions that vanish to order two (that is, the dual to the
\emph{Lie algebra}). 
\end{example}

\subsection{Stacks associated to ring spectra}
\label{sec:stackring}
We will also need a means of extracting stacks from ring spectra which may not
be as well-behaved as $MU$ (cf. \cite[Ch. 7]{TMF}).  
Suppose $X$ is a homotopy commutative ring spectrum. Then the above 
cosimplicial diagram
\[ X \wedge MU \rightrightarrows X \wedge MU \wedge MU \triplearrows \dots  \]
is a diagram of ring spectra itself, and the associated
diagram of homotopy groups is a diagram of commutative rings if $MU_* X$ is
evenly graded. Consequently, the associated sheaf $\mathcal{F}(X) =
\mathcal{F}_{\mathrm{even}}(X)$ on the stack $M_{FG}$
is a sheaf of commutative rings, and can be used to present another stack,
affine over $M_{FG}$.

\begin{definition} 
We write $\st(X)$ for the stack built in the above manner. 
Equivalently, $\st(X)$ is the stack associated to the Hopf algebroid $(MU_* X,
MU_*(MU \wedge X))$.
\end{definition} 

\begin{example} \label{complexorstack}
Consider $X = MU$. Recall that $\spec (\pi_*( MU \wedge MU))$ classifies a pair of formal group laws
with a  strict isomorphism between them. The scheme $\spec (\pi_* (MU \wedge MU
\wedge MU))$
corresponds to a triple of formal group laws with strict isomorphisms between them. It follows that the
associated stack is equivalent to the $\mathbb{G}_m$-quotient of
$\spec (L)$, where $L$ is the Lazard ring. 
More generally, whenever $R$ is a complex-orientable ring spectrum with
$\pi_* R$ evenly graded, the stack
$\st(R)$ is the $\mathbb{G}_m$-quotient of $\spec (R_*)$. 
It follows in particular that for such $R$, the sheaf
\( \mathcal{F}(R)   \) is the push-forward of the structure sheaf under the map
$\spec( R_*)/\mathbb{G}_m \to M_{FG}$. 
\end{example} 

\subsection{Even periodic ring spectra}

Let $X$ be a spectrum. As we saw,  $X$ defines  quasi-coherent sheaves
$\mathcal{F}_{\mathrm{odd}}(X), \mathcal{F}_{\mathrm{even}}(X)$ on the moduli  stack
$M_{FG}$ of formal groups. These come from $MU_*(X)$ together with the comodule
structure over the Hopf algebroid $(MU_*, MU_* MU)$ and the grading.
Equivalently, we can consider the periodic complex bordism spectrum $MP =
\bigvee_{i \in \mathbb{Z}}
\Sigma^{2i} MU$. 
In this case, we can write $MU_*(X) = MP_0(X) \oplus MP_1(X)$ and the even and
odd parts  of the grading correspond to the two summands. 

On the flat site of $M_{FG}$, there is a classical topological interpretation of
these sheaves.

\begin{definition}[\cite{AHS}] 
A homotopy commutative ring spectra $E$ is \emph{even periodic} if $\pi_i E =
0$ for $i$ odd, and if $\pi_2 E$ is an invertible module over $\pi_0 E$ with
the property that the multiplication map 
\( \pi_2 E \otimes_{\pi_0 E} \pi_{-2} E \to  \pi_0 E,  \)
is an isomorphism. \end{definition} 

Given an even periodic ring spectrum $E$, the formal scheme $\mathrm{Spf}
E^0(\mathbb{CP}^\infty)$ is a formal group over $E_0$ (see \cite{AHS}), and is
classified by a morphism $q\colon \spec (E_0) \to M_{FG}$. 
If $q\colon \spec (E_0 ) \to M_{FG}$ is flat, then $E$ is called a
\emph{Landweber-exact theory} and one has functorial isomorphisms:
\[  q^* (  \omega^j \otimes  \mathcal{F}_{\mathrm{even}}(X))  \simeq
E_{2j}(X) , \quad q^* ( \omega^j \otimes \mathcal{F}_{\mathrm{odd}}(X)) \simeq
E_{2j+1}(X). \]
In particular,  for a Landweber-exact theory, $E_*(X)$ can be recovered from
the sheaves $\mathcal{F}_{\mathrm{even}}(X), \mathcal{F}_{\mathrm{odd}}(X)$.
The stack associated to the ring spectrum $E$ is precisely $\spec (E_0)$. 

Conversely, given a ring $R$ and a flat morphism $q \colon \spec (R) \to M_{FG}$,
the functor
\[ X \mapsto  \bigoplus_j q^*\left( \omega^j \otimes
( \mathcal{F}_{\mathrm{even}}(X) \oplus
\mathcal{F}_{\mathrm{odd}}(X)) \right) ,  \]
defines a multiplicative homology theory, representable by a Landweber-exact
ring spectrum. 
The result is a \emph{presheaf of multiplicative homology theories} on the
flat site of $M_{FG}$. 
One reason this point of view is so useful is the following 
criterion for
flatness over $M_{FG}$ given by the {Landweber exact functor theorem}. 
For a discussion of the Landweber exact functor theorem and its interpretation
via the language of stacks (due to Hopkins), we refer to  
\cite{Sheaves} and \cite[Lecture 16]{chromatic}. 

\subsection{Topological modular forms}
Although $M_{FG}$ is a very large stack, there are smaller stacks that can be
used to approximate it. 
\begin{definition}
Let $\mell$ be the moduli  stack of generalized elliptic curves. 
In other words, $\mell$ assigns to each commutative ring $R$ the groupoid of all pairs
 $(\pi \colon C
\to \spec (R), e\colon \spec (R) \to C)$, where $\pi, e$ are such that: 
\begin{enumerate}
\item 
$\pi$ is a proper, flat morphism of finite presentation. 
\item  The 
fibers of $\pi$ have arithmetic genus one, and they are either smooth curves or curves
with a single nodal singularity. 
\item  $e$ is a section of $\pi$ whose
image is contained in the smooth locus of $\pi$. 
\end{enumerate}
\end{definition}

There is a flat map  $\mell \to M_{FG}$ which sends each generalized elliptic curve  to its formal completion at the identity, which acquires the structure
of a formal group. Consequently, any spectrum $X$ defines by pullback 
quasi-coherent sheaves on $\mell$ as well; we will denote these too by
$\mathcal{F}_{\mathrm{even}}(X),
\mathcal{F}_{\mathrm{odd}}(X)$. 

As before, the quasi-coherent sheaf $\mathcal{F}(X)$ on $\mell$ has a
topological interpretation. Given a flat map $q\colon \spec (R) \to \mell$
classifying an elliptic curve $C \to \spec (R)$, one can
construct an \emph{elliptic spectrum} $E$: in other words,   $E$ is even periodic with $E_0 = R$, and there is an
isomorphism of formal groups
\[ \mathrm{Spf} E^0(\mathbb{CP}^\infty) \simeq \hat{C},  \]
between the formal group of $E$ and the formal completion of $C$. 
In this case, one has $E_0( X)  \simeq q^*(\mathcal{F}_{\mathrm{even}}(X))$ and 
$E_1(X) \simeq q^* ( \mathcal{F}_{\mathrm{odd}}(X))$.
Moreover, there is  a functorial isomorphism
\[ E_{2j}(\ast) \simeq \omega^j, \quad E_{2j+1}(\ast) = 0,   \]
where $\omega$ is the line bundle on $\mell$ obtained by pulling back
$\omega$ on $M_{FG}$. 

The work of Goerss, Hopkins, Miller, and Lurie shows that, when we restrict to
\emph{\'etale} affines over $\mell$, $E$ actually can be
taken to be an $E_\infty$-ring spectrum, and that the above construction
is  can be made to be functorial: it defines a sheaf $\otop$ of $E_\infty$-ring spectra on the \'etale site of $\mell$. 
We refer to 

\begin{theorem}[Goerss, Hopkins, Miller, Lurie]
\label{blackbox}

\begin{enumerate}
\item (Existence)
There is a sheaf  $\otop$ of $E_\infty$-ring spectra on the \'etale site
of $\mell$, such that for
$\spec (R) \to \mell$ an affine \'etale open classifying an elliptic curve $C
\to \spec (R)$, the $E_\infty$-ring $\otop(\spec (R))$ is an elliptic
spectrum corresponding  to $C$. 
\item (Gap theorem)
Moreover,
$\pi_j( \Gamma(\mell, \otop)) = 0$ for $- 21 < j < 0$. 
\end{enumerate}
\end{theorem}

We refer to \cite[Lec. 11--12]{TMF} for a treatment of the construction of the
sheaf $\otop$ and to \cite{computation} and \cite[Lecture 13]{TMF} for some of the
computations of the homotopy groups. See also the masters thesis of Johan Konter
\cite{konter} for the latter. 

They define
\[ \Tmf =   \Gamma(\mell, \otop) ,  \quad \quad \tmf = \tau_{\geq 0}(\Tmf),\]
where $\tau_{\geq 0}$ denotes the connective cover. 
In other words, $\Tmf$ is the $E_\infty$-ring spectrum of global sections of $\otop$, constructed as a
homotopy limit of the associated elliptic spectra as $\spec (R) \to \mell$ ranges over \'etale
morphisms. 
Since $\Tmf$ is constructed as a homotopy limit, we have a descent spectral
sequence
\[ H^i( \mell, \pi_{2j} \otop) 
 =  H^i(\mell, \omega^j)
\implies \pi_{2j-i} \Tmf.  \]
More generally, if $X$ is a spectrum, then  we get a spectral sequence
\begin{equation}\label{tmfss} H^i(\mell, \pi_j( \otop \wedge X)) \implies \pi_{j - i } (\Tmf
\wedge X).  \end{equation}
Here we use the fact that $\Tmf \wedge X \simeq \Gamma( \mell, \otop \wedge
X)$; the fact that this holds for any $X$, not necessarily finite, follows
because $\mell \to M_{FG}$ is tame (cf. \cite[Th. 4.14 and 7.2]{MMaffine}).
A knowledge of $\mathcal{F}_{\mathrm{even}}(X)$ and
$\mathcal{F}_{\mathrm{odd}}(X)$, or equivalently  of $\pi_0(\otop
\wedge X)$ and $\pi_1(\otop \wedge X)$, is the information necessary to
identify the $E_2$-page of the  
spectral sequence. 

The following case, which is that of interest to us, offers a simplification of
the spectral sequence: 
\begin{definition}[\cite{AHS}] A connective spectrum $X$ is \emph{even} if
$H_*(X; \mathbb{Z})$ is free and concentrated in even dimensions. We can make a
similar definition for a $p$-local spectrum, for a prime $p$. 
\end{definition} 

In the even case, the sheaf $\mathcal{F}_{\mathrm{even}}(X)$ on $M_{FG}$ can be interpreted in the
following way: for a flat morphism $q \colon \spec (R) \to M_{FG}$,  the $R$-module $q^*
\mathcal{F}_{\mathrm{even}}(X)$ is identified with $E_0(X)$ for $E$ the Landweber-exact,
even-periodic ring spectrum associated with $q$. 
It follows from the (degenerate) Atiyah-Hirzebruch spectral sequence that
$E_1(X) =0 $, and in particular $\mathcal{F}_{\mathrm{odd}}(X) =0$, 
$\mathcal{F}(X) = \mathcal{F}_{\mathrm{even}}(X)$.  We find that
$\mathcal{F}(X)$ can be simply viewed as a sheaf on $M_{FG}$ or $\mell$. 
In this case, the Adams-Novikov spectral sequence can be written as 
\[ H^i(M_{FG}, \mathcal{F}(X) \otimes \omega^j) \implies \pi_{2j-i}(X) , \]
and the descent spectral sequence can be written as
\begin{equation} \label{dsseven} H^i(\mell, \omega^j \otimes \mathcal{F}(X)) \implies \pi_{2j-i}( \Tmf
\wedge X).   \end{equation}

\section{A vector bundle on $M_{cub}$}

In this section, we will introduce the moduli stack $M_{cub}$ of cubic curves
and exhibit an eight-dimensional vector bundle at the prime 2. 

\subsection{The stack $M_{cub}$}
For our purposes, it will be convenient to work over the larger stack
of all cubic curves.  The results on cubic
curves that we need can be found in 
\cite{formulaire}. 

\begin{definition}
Given a scheme $S$, a \emph{cubic curve} over $S$ is a map $p\colon E \to S$ which is flat and proper of finite
presentation, together with a
section $e\colon S \to E$ whose image is contained in the smooth locus of $p$. The
geometric fibers
 of $p$ are required to be 
reduced, irreducible curves of arithmetic genus one.
We denote by $M_{cub}$ the stack which assigns to each commutative ring $R$ the
groupoid of cubic curves over $\spec (R)$. 
Below, we will recall a Hopf algebroid presentation of $M_{cub}$. 
\end{definition}

\begin{definition}
There is a line bundle $\omega$ on the stack $M_{cub}$, which assigns to a cubic 
curve $p\colon C \to \spec (R), e\colon \spec (R) \to C$ the $R$-module of sections of
$\mathcal{O}_C(-e)/\mathcal{O}_C(-2e)$: that is, the cotangent space along the
zero section $e$, or the dual to the Lie algebra. \end{definition}

If $p\colon E \to S$ is a cubic curve, there are three
possibilities for each geometric fiber: it can be an elliptic curve, a nodal
cubic in $\mathbb{P}^2$, or a
cuspidal cubic in $\mathbb{P}^2$ (isomorphic to the projective closure of the
curve defined by $y^2 = x^3$).
In the first two cases, there are no infinitesimal
automorphisms. However, the multiplicative group $\mathbb{G}_m$ acts on the
cuspidal curve. In particular, $M_{cub}$ is only an Artin stack,
which contains the Deligne-Mumford stack $\mell$ as an open substack. 
The complement of $\mell$ in $M_{cub}$ is given by the vanishing locus of the
{modular forms} $c_4 \in H^0(M_{cub}, \omega^4)$ and $\Delta \in H^0(M_{cub},
\omega^{12})$ (cf. \cite{formulaire} for expressions for these). 

Zariski locally on $S$,  a cubic curve 
can be described as a subscheme of $\mathbb{P}^2_S$ cut out by a cubic equation
\begin{equation} \label{cubic} y^2 +
a_1 xy + a_3 y = x^3 + a_2 x^2 + a_4 x + a_6. \end{equation} In order to
show this, one has to
choose coordinates $x,y$ (which are chosen as sections of  appropriate line
bundles). Consequently, one can write down an explicit presentation of this
stack via a Hopf algebroid. We sketch this below. 

Let $E \to \spec (R)$ be a cubic curve, given by a cubic equation \eqref{cubic}. 
Given one choice of $x,y$, the collection of other choices of coordinates
is parametrized by \begin{gather}\label{change} x = u^2 x' + r \\
y = u^3 y' + su^2 x' + t
\end{gather}
where $u \in R^{\ast}$ and $r,s,t \in R$. These are the isomorphisms between
Weierstrass curves. 

In particular, we can represent $M_{cub}$ as a Hopf algebroid (the
\emph{Weierstrass Hopf algebroid}) over the ring
$\mathbb{Z}[a_1, a_2, a_3, a_4, a_6]$. 
The left and right units come from the transformation laws of a cubic equation,
and the comultiplication comes from composition of isomorphisms. 

Suppose given a cubic curve \eqref{cubic}, and suppose one makes a change of
coordinates as in \eqref{change}. Then the new Weierstrass cubic, in
coordinates $x', y'$, has the form
\[   y'^2 +
a'_1 x'y' + a_3' y '= x'^3 + a_2' x'^2 + a'_4 x' + a'_6, \]
where:
\begin{gather} 
ua_1' = a_1 + 2s \\
u^2 a_2' = a_2 - sa_1 + 3r - s^2 \\
u^3 a_3' = a_3 + ra_1 + 2t \\
u^4 a_4' = a_4 - sa_3 + 2a_2 r - (t + rs)a_1 + 3r^2 - 2st \\
u^6 a_6' = a_6 + ra_4 + r^2 a_2 + r^3 - ta_3 - t^2 - rta_1 
\end{gather} 
The stack $M_{cub}$ is presented either by the Hopf algebroid $$(\mathbb{Z}[a_1,
\dots, a_6], \mathbb{Z}[a_1, \dots, a_6][u^{\pm 1}, r,  s, t]),$$ or by the
\emph{graded} Hopf algebroid $$(\mathbb{Z}[a_1,
\dots, a_6], \mathbb{Z}[a_1, \dots, a_6][r, s, t]), \quad \quad |a_i| = 2i  \text{
and } 
|r| =  4, |s| = 2, |t| = 6.$$ The grading is doubled as in topology.  

Given a cubic curve $E \to
\spec (R)$, there is induced a structure of commutative group scheme on the smooth
locus $E^{\circ} \to \spec (R)$ (see Proposition 2.7 of \cite{DR}). 
In fact, $E^{\circ}$ can be described as the \emph{relative Picard scheme}
$\mathrm{Pic}^0_{E/\spec (R)}$; see Theorem 2.6 of \cite{DR}. 
This is also described in Proposition 2.5 of III in \cite{silverman}.
In particular, we can take the formal completion at the zero section and  get a \emph{formal group} over $R$. This gives a morphism of stacks
\[ M_{cub} \to M_{FG}.  \]

\begin{example} 
The geometric points of $M_{cub}$ fall into four types in characteristic
$p>0$. There are the \emph{ordinary} elliptic curves, which map to 
height one formal groups, as do the \emph{nodal} elliptic curves. There are
\emph{supersingular} elliptic curves, which map to height two formal groups.
Finally, there is the \emph{cuspidal cubic}, which maps to the additive formal group
(of infinite height). 
\end{example} 

Consequently, given a spectrum $X$, defining sheaves
$\mathcal{F}_{\mathrm{even}}(X), \mathcal{F}_{\mathrm{odd}}(X)$ on
$M_{FG}$, we can pull back to define a sheaf (which we will still denote by
the same notation) on $M_{cub}$.  If $X$ is
an \emph{even} spectrum, then $\mathcal{F}(X) = \mathcal{F}_{\mathrm{even}}(X)$ will be a vector bundle
on $M_{cub}$ too. 
Note, however, that the map $M_{cub} \to M_{FG}$ is no longer flat, unlike the
map $\mell \to M_{FG}$. 

We note that the line bundle $\omega$ on $M_{FG}$ pulls back to the line bundle
$\omega$ on $M_{cub}$; given an elliptic curve $p \colon C \to \spec (R), e
\colon \spec (R)
\to C$, the $R$-module $\mathcal{O}_C(-e)/\mathcal{O}_C(-2e)$ is also the
\emph{Lie algebra} of the formal group. In particular, it can also be described
as associated to the graded comodule $\mathbb{Z}[a_1, \dots , a_6]$ over the
Weierstrass Hopf algebroid with the
grading shifted by $2$. 

\subsection{An eight-fold cover of $M_{cub}$}
Henceforth, in this section, we work localized at 2 throughout: in particular, stacks such as
$M_{cub}$ will really mean $M_{cub} \times_{\spec (\mathbb{Z})} \spec
(\mathbb{Z}_{(2)})$. 
We will exhibit an eight-dimensional vector bundle on $M_{cub}$ (which we will
later see corresponds to a finite spectrum) by producing a finite flat cover
$p\colon T \to M_{cub}$, for $T$ a simpler stack. The associated vector bundle will
be $p_*\mathcal{O}_{T}$. We refer to Section~\ref{levestructureinterpretation} 
below for a discussion of the modular interpretation. 

Namely, we take for $T$ the stack-theoretic quotient
\[ T = \spec (\mathbb{Z}_{(2)}[\alpha_1, \alpha_3]) / \mathbb{G}_m,  \]
where the $\mathbb{G}_m$-action corresponds to the grading of the ring
$\mathbb{Z}_{(2)}[\alpha_1, \alpha_3]$ with $|\alpha_1 | = 2, |\alpha_3| =6$,
where the grading is \emph{doubled} in accordance with topology. 
In other words, we can think of $T$ as the stack associated to the prestack  sending a ring $R$ to the groupoid of pairs of elements
$(\alpha_1, \alpha_3) \in R$, with 
\[ \hom ( (\alpha_1, \alpha_3), (\alpha_1', \alpha_3')) = \left\{u \in R^*
 :  u \alpha_1 = \alpha_1', u^3 \alpha_3 = \alpha_3'\right\} . \]

To produce the cover $p\colon T \to M_{cub}$, observe first that there is a map $\spec
(\mathbb{Z}_{(2)}[\alpha_1, \alpha_3 ]) \to M_{cub}$ classifying the cubic
curve
\[ y^2 + \alpha_1 xy + \alpha_3 y = x^3.  \]
Observe that $\mathbb{G}_m$ also acts on the cubic curve in a corresponding
fashion as in the action on $\mathbb{Z}_{(2)}[\alpha_1, \alpha_3]$. 
Namely, given an invertible element $u$, we have the transformation
\( x \mapsto u^2x,  \ y \mapsto  u^3 y  \)
from the curve $y^2 + \alpha_1 xy + \alpha_3y = x^3$ into the curve $y^2 + u
\alpha_1 xy +
u^3 \alpha_3 y = x^3$. 
These isomorphisms allow us to
produce the morphism of stacks $T \to
M_{cub}$, as desired. 

\begin{proposition} \label{cover}The map $p\colon T \to M_{cub}$ is
representable and is
a finite, flat cover of rank eight. 
\end{proposition}

Proposition~\ref{cover} will be proved in the next subsection. 
\subsection{Verification of the rank eight cover}
The goal of this subsection is to establish Proposition~\ref{cover}. 
To check that a morphism $\mathfrak{X} \to \mathfrak{Y}$ of stacks is finite and flat, we just need
to show that for every map $\spec (R) \to \mathfrak{Y}$, the pullback $\spec
(R)
\times_\mathfrak{Y} \mathfrak{X}$
is a finite flat cover of $\spec
(R)$. 
In our case, we need to show that for every morphism $\spec (R) \to M_{cub}$,
when one forms the pullback square,
\[ \xymatrix{
P \ar[d] \ar[r] & T\ar[d]  \\
 \spec (R) \ar[r] &  M_{cub}
},\]
the fiber product $P$ is an affine scheme, corresponding to the $\spec$ of an algebra which is a
finite, flat $R$-module of rank eight. 

Let us first identify the pullback $\spec( R) \times_{M_{cub}} \spec
(\mathbb{Z}_{(2)}[\alpha_1, \alpha_3])$ concretely. Suppose that $\spec (R) \to M_{cub}$
 classifies a
cubic curve in $\mathbb{P}^2_R$ cut out by
the Weierstrass equation
$$y^2 + a_1 xy + a_3 y = x^3 + a_2x + a_4 + a_6, \quad a_i \in R.$$ 
By passing to a Zariski cover of $\spec (R)$, we can always assume this. 
Then, to form the pullback of stacks, we have to consider the scheme parametrizing
changes of coordinates (over $R$) which will transform this equation into an
equation 
\[   y'^2 +
a'_1 x'y' + a_3' y '= x'^3 + a_2' x'^2 + a'_4 x' + a'_6, \]
where $a_2', a'_4, a'_6 = 0$. In other words, looking back at the previous
formulas, $\spec (R) \times_{M_{cub}} \spec (\mathbb{Z}_{(2)}[\alpha_1,
\alpha_3])$
is the scheme parametrizing elements $u, r, s, t$ (with $u$ invertible)
satisfying the equations 
\begin{gather}  \label{threerelns}
0 = a_2 - sa_1 + 3r - s^2  \\
0 = a_4 - sa_3 + 2a_2 r - (t + rs)a_1 + 3r^2 - 2st \label{threerelns2} \\
0 = a_6 + ra_4 + r^2 a_2  + r^3 - ta_3 - t^2 - rta_1 \label{threerelns3}
\end{gather} 
In other words,  it is $R[u^{\pm 1}, r, s, t]$ modulo the above relations.
Taking the pullback $P = \spec (R) \times_{M_{cub}} T$ corresponds to taking the
$\mathbb{G}_m$-quotient: in other words, we just have to ignore $u$. 

We find that $P$ is a closed subscheme of affine space $\mathbb{A}^3_{R}$ cut
out by the three equations above;
by the first relation, $r$ is determined in terms of $s$, and $P$ is even a
closed subscheme of $\mathbb{A}^2_R$ cut out by two equations. 
Our goal is to show that it is finite flat over $\spec (R)$, of rank eight. 
Let us first show finiteness. 
\begin{lemma} 
Let $R$ be a ring, and let $a_1, a_2, a_3, a_4, a_6 \in R$. Then the quotient
of $R[r, s, t]$ by the above three relations \eqref{threerelns},
\eqref{threerelns2}, \eqref{threerelns3} is  a finite
$R$-module. 
\end{lemma} 
\begin{proof}
We can do this in the universal case where $R = \mathbb{Z}_{(2)}[a_1, \dots,
a_6]$. Let $I$ be the ideal generated by the above relations. 
In this case, observe that $R$ is naturally a \emph{graded} ring (where we
set $|a_i|
= 2i$), and the ring
$R[r,s,t]/I$ is a graded ring too, if we set
\[ |r| = 4, \quad |s| = 2, \quad |t| = 6.  \]
In other words, the relations defining $I$ are homogeneous. 
In view of Nakayama's lemma, we can now prove finiteness by taking the
quotient by the
augmentation ideal: that is, by  working over $\mathbb{Z}_{(2)}$, and by setting
all the $a_i = 0$ (so that we have the cuspidal cubic $y^2 = x^3$, and the
grading of $R[r, s, t]/I$ comes from this grading). Then the
relevant ring is 
the quotient of $\mathbb{Z}_{(2)}[r, s, t]$ by the relations
\begin{gather} 3r - s^2 = 0  , \quad 3r^2 - 2st = 0 , \quad r^3 - t^2  = 0 . 
\end{gather}
Alternatively, this is the quotient of $\mathbb{Z}_{(2)}[s, t]$ under the
relations
\[ 3( s^2/3)^2 =  2st , \quad (s^2/3)^3 = t^2  .\]
Since this is a graded $\mathbb{Z}_{(2)}$-module each of whose graded pieces is
finitely generated, we may as well prove finiteness after tensoring with
$\mathbb{Z}/2$, by the ungraded version of Nakayama's lemma. 
We then get the $\mathbb{Z}/2$-algebra $\mathbb{Z}/2[s, t]/(s^4, t^2)$, which
is evidently finite over $\mathbb{Z}/2$.
\end{proof}

If $R$ is any ring, $a_i \in R$ for $i = 1, 2, \dots, 6$, and we form the
quotient $R[r, s, t]/I$ as above, the above
proof also shows that there is  a set of eight generators of the $R$-module
$R[r, s, t]/I$, given by
\[ \left\{1, s, s^2, s^3, t, st, s^2 t, s^3 t\right\} . \]
Namely, we just need to prove this in the universal case
$\mathbb{Z}_{(2)}[a_1, a_2, a_3, a_4, a_6]$, which reduces by the
Nakayama-type argument above to the case of the cuspidal cubic over
$\mathbb{Z}/2$. We have seen 
that the above elements generate in that case and in fact form a basis. 

Now that we know that the map $T \to M_{cub}$ is finite, flatness is automatic.
In fact, we saw in the proof that for any map $\spec (R) \to M_{cub}$, the
pullback $T \times_{M_{cub}} \spec (R)$ was, up to Zariski localizing on $R$, of the form  $R[r,s,t]/I$, where $I$ was an ideal generated by three
elements.
This implies flatness by the following lemma. 

\begin{lemma} 
Let $X$ be a noetherian scheme, and let $Z \subset \mathbb{A}^n_X$ be a closed subscheme
locally cut out by $n$ equations. Suppose the fibers $Z_x, x \in X$ are
zero-dimensional. Then $Z \to X$ is flat. 
\end{lemma} 
\begin{proof} 
In fact, we may suppose $X = \spec (R)$, for $(R, \mathfrak{m})$ a local noetherian ring. 
We need to show that the local rings $\mathcal{O}_{Z, z}$ are flat $R$-modules
for each $z \in Z$ lying over the maximal ideal of $R$. The ring
$\mathcal{O}_{Z, z}$ is obtained from a localization of $R[t_1, \dots, t_n]$ at a
maximal ideal $\mathfrak{M} \supset \mathfrak{m}$, by taking the quotient by
$n$ equations $f_1,  \dots, f_n$. 
Observe moreover that the local ring $S = R[t_1, \dots, t_n]_{\mathfrak{M}} $ is
flat over $R$, 
and $S/(f_1, \dots, f_n)$ is the localization of a finite $R$-module. 
Let $k$ be the residue field of $R$. Then $S \otimes_R k$
is an $n$-dimensional regular local ring and $(S \otimes_R k)/(f_1, \dots, f_n)$
is artinian, so $f_1, \dots, f_n$ form a regular sequence on $S \otimes_R
k$. It follows now from a Nakayama argument (cf. \cite[Prop 15.1.16]{EGA})  that $f_1, \dots,  f_n$ are regular on $S$
itself, and that $ \mathcal{O}_{Z, z} = S/(f_1, \dots, f_n)$ is a flat $R$-module. 
\end{proof}

Finally, let us show that the map $p\colon \spec (\mathbb{Z}_{(2)}[\alpha_1,
\alpha_3])/\mathbb{G}_m \to M_{cub}$ is a \emph{cover}, i.e., that it is surjective.  
Surjectivity follows because $p$ is finite flat, so that the image is both
open and closed; however, $M_{cub}$ admits a cover by $\spec
(\mathbb{Z}_{(2)}[a_1, a_2, a_3, a_4, a_6])$ and is thus connected. 

The upshot of all this is that the map 
\( p\colon T = \spec (\mathbb{Z}_{(2)}[\alpha_1, \alpha_3])/\mathbb{G}_m \to M_{cub}  \)
is an eight-fold flat cover, and consequently the push-forward of the structure
sheaf gives us an eight-dimensional vector bundle $\mathcal{V}$ on $M_{cub}$. Our goal is to
show that this vector bundle is realized by an 8-cell complex. 

\section{Calculation of $\Tmf_*(DA(1))$}

In this section, we recall the
2-local complex $DA(1)$ and show that the vector
bundle it induces on $\mell$ is the one constructed algebraically in the
previous section. 
We will do this by producing a map $\mathcal{F}(DA(1)) \to
\mathcal{F}(MU)$ and a map from $\mathcal{F}(MU)$ to the vector
bundle of the previous section, over $M_{cub}$. We will check that the
composite is an isomorphism over the cuspidal cubic over $\mathbb{Z}/2$. 
This analysis leads to the 
computation of $\Tmf_*(DA(1))$ by the descent spectral sequence. 
In the final subsection, we will discuss the analog at odd primes. 

\subsection{The complex $DA(1)$}
In this section, we describe a 2-local complex with eight-dimensional homology. 

We start by describing the question mark complex, a variant of
which is constructed in
\cite[Lem. 7.2]{minimalatlases}.
First, let $\nu\colon S^3 \to S^0$ and $\eta\colon S^1 \to S^0$ be the usual Hopf
maps. 
We have $\eta \nu = 0 $. 
We draw the cofiber sequence for $\eta$, which runs $S^1 \stackrel{\eta}{\to}
S^0 \to \Sigma^{-2} \mathbb{CP}^2 \to S^2 \to \dots$, and consider the diagram: 
\[
\xymatrix{
 & & & S^5 \ar[d]^{\Sigma^2 \nu} \ar@{-->}[ld]^t\\
S^1 \ar[r]^{\eta} &   S^0 \ar[r] &  \Sigma^{-2} \mathbb{CP}^2   \ar[r]
&   S^2 \ar[d] 
\ar[r]^{\Sigma \eta} &  S^1  \\
 & &   & \Sigma^{-2} \mathbb{HP}^2.
}
 \]
 The map $t$ drawn as a dotted arrow exists because $\eta \nu = 0$; it is
 even unique as $\pi_5(S^0) = 0$. 

 \begin{definition}
The \emph{question mark complex} $Q$ is defined to be the cofiber of $t$. 
\end{definition}

A simple calculation shows that $H^*(Q; \mathbb{Z}/2)$ is three-dimensional,
with a basis given by the element $x_0$ in degree zero, and $\mathrm{Sq}^2 x_0$
and $\mathrm{Sq}^4 \mathrm{Sq}^2 x_0$. 

\begin{definition} 
Let $\mathbb{D}Q = F(Q, S^0)$ be the Spanier-Whitehead dual to $Q$. 
We define the (2-local) complex $DA(1)$ to be the six-fold
suspension of the cofiber of the coevaluation map $S^0 \to Q \wedge \mathbb{D}Q$.
\end{definition} 

The composite of the coevaluation and evaluation maps
\( S^0 \to Q \wedge \mathbb{D}Q \to S^0  \)
is multiplication by  the Euler characteristic  $\chi(Q) = 3$. Since we are working 2-locally, we
find that there is  a splitting
\( Q \wedge \mathbb{D}Q \simeq S^0 \vee \Sigma^{-6}DA(1).  \)

Note that $DA(1)$ is an \emph{even} 2-local spectrum. 
A computation shows that the cohomology $H^*(DA(1); \mathbb{Z}/2)$ is a free module
over the eight-dimensional algebra $D\mathcal{A}(1) \subset
\mathcal{A}/\mathcal{A} \mathrm{Sq}^1 \mathcal{A}$ generated by
$\mathrm{Sq}^2, \mathrm{Sq}^4$; the cohomology is drawn in Figure~\ref{fig:DA1}. 
Its homology, as a comodule over $\A$, can be described as 
$\mathbb{Z}/2\left\{1, \xi_1^2, \xi_1^4, \xi_1^6\right\} \otimes
\mathbb{Z}/2\left\{1, \zeta_2^2\right\} \subset \A$.

{\centering
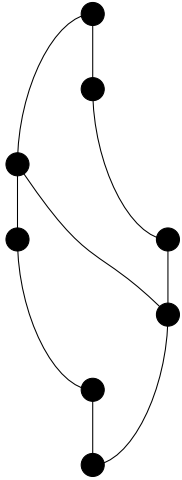
\begin{figure} {\small \begin{tikzpicture}
\tikzstyle{every node}=[draw, fill=black]
\draw  (0,0) -- (0, 1) node[circle]{}; 
\fill[fill=black] (0, 0) node[circle]{};
\draw (0, 0) .. controls (.5,0)  and (1,1) .. (1, 2) node[circle]{};
\draw (0, 1) .. controls (-.5,1)  and (-1,2) .. (-1, 3) node[circle]{};
\draw (1, 2) -- (1,3) node[circle]{};
\draw (-1, 3) -- (-1, 4) node[circle]{}; 
\draw (1, 2) .. controls (0,3) and (0, 2.5)  .. (-1, 4); 
\draw (1,3) .. controls (0.5, 3) and (0, 4) .. (0, 5) node[circle]{};
\draw (0, 5) -- (0, 6);
\draw (-1, 4) .. controls  (-1, 5) and (-0.5, 6) .. (0, 6) node[circle]{}; 
\end{tikzpicture}}
\caption{The cohomology of $DA(1)$. This depicts only the action of
$\mathrm{Sq}^2, \mathrm{Sq}^4$; there is also a nontrivial $\mathrm{Sq}^8$
which is not shown.}
\label{fig:DA1}
\end{figure}}

\begin{remark} 
The complex is so named because its cohomology is free over the subalgebra of
the Steenrod algebra generated by $\mathrm{Sq}^2$ and $\mathrm{Sq}^4$, which
doubles the subalgebra $\mathcal{A}(1) \subset \mathcal{A}$ generated by $\mathrm{Sq}^1,
\mathrm{Sq}^2$. 
\end{remark} 

Since $DA(1)$ is an even spectrum, the Atiyah-Hirzebruch spectral sequence
for $MU^*(DA(1))$ degenerates, and we can find a morphism
\( DA(1) \to MU  \)
which induces an isomorphism on $\pi_0$.  This map is not unique, but we have
specified its image in mod two homology after composing with the projection
$MU \to BP$. 
To see this, we need to first recall some notation.

\begin{conv}
We let $\A = \pi_*( H \mathbb{Z}/2 \wedge H\mathbb{Z}/2)$ denote the (mod 2) dual Steenrod algebra. Then we have
\[ \A = \mathbb{Z}/2 [\xi_1, \xi_2, \dots ], \quad |\xi_i| = 2^i - 1,  \]
using the standard notation (cf. \cite[Ch. 6]{MosherTangora} for a textbook
reference). 
We will let $\zeta_i$ denote the Hopf conjugate of $\xi_i$. 
\end{conv}

We recall that the map $BP \to H \mathbb{Z}/2$ induces an injection in mod 2
homology, and we get $H_*(BP; \mathbb{Z}/2) = \mathbb{Z}/2[\xi_i^2] =
\mathbb{Z}/2[\zeta_i^2]
\subset \A$. Compare \cite[Ch. 4]{ravenel} for a textbook reference. The image of $H_*(DA(1);
\mathbb{Z}/2)$ in $H_*(BP; \mathbb{Z}/2)$ is 
spanned by $\left\{\xi_1^a \zeta_2^b\right\}$ for $0 \leq a \leq 3$
and $0 \leq b \leq 1$.

Both $DA(1)$ and $MU$ define quasi-coherent sheaves on $M_{FG}$, and thus on $\mell$
and even $M_{cub}$; these are denoted $\mathcal{F}(DA(1))$ and $\mathcal{F}(MU)$.
The map $DA(1) \to MU$ defines an injection of sheaves
\[ \mathcal{F}(DA(1)) \to \mathcal{F}(MU).  \]
Our goal will be to produce a map $\mathcal{F}(MU) \to \mathcal{V}$, for
$\mathcal{V}$ the 
eight-dimensional vector bundle constructed in the previous section, such that
the composite is an isomorphism. We will construct the map, and  check that it is an
isomorphism on the cuspidal curve over
$\mathbb{Z}/2$. 

\subsection{Construction of a map}
Let us start by describing the sheaf 
$\mathcal{F}(MU)$ on the stack $M_{FG}$. This assigns to a formal group  over a ring $R$ the ring
parametrizing
coordinates on this formal group, modulo action of $\mathbb{G}_m$. 
Namely, let $M_{FG}^{coord}$ be the moduli stack of formal groups together
with a  coordinate (i.e., an isomorphism of formal schemes with the formal
affine line $\widehat{\mathbb{A}^1}$), so that $M_{FG}^{coord} $ is simply the
spectrum of the Lazard ring. 
It parametrizes formal group \emph{laws} (and no isomorphisms).

There is a
$\mathbb{G}_m$-action on $M_{FG}^{coord}$, corresponding to twisting a coordinate:
this induces the usual grading of the Lazard ring. 
There is a morphism of stacks
\[ q\colon M_{FG}^{coord}/\mathbb{G}_m \to M_{FG}.   \]
Then $\mathcal{F}(MU) = q_*(\mathcal{O})$. 

Consider next the pullback diagram
\[  
\xymatrix{
M_{cub}^{coord}/\mathbb{G}_m \ar[d]^{q'} \ar[r] & M_{FG}^{coord}/\mathbb{G}_m
\ar[d]^q  \\
M_{cub} \ar[r] &  M_{FG}
}.
\]
Here $M_{cub}^{coord}/\mathbb{G}_m$ is the stack parametrizing cubic curves
together with a coordinate on the formal group, modulo $\mathbb{G}_m$-action.
We have: 

\begin{theorem}[{Cf.  
\cite[Cor. 8.9]{EllSH}}]
\label{mcoordcub}
There is an equivalence of stacks $$
M_{cub}^{coord}/\mathbb{G}_m = 
\spec
(\mathbb{Z}_{(2)}[a_1, \dots, a_6, e_4, e_5, \dots ])/\mathbb{G}_m.$$ \end{theorem} 

The $\left\{a_i\right\}$ are the choice of coefficients for a Weierstrass equation. In fact, the choice of a coordinate modulo degree five on the formal group of a cubic curve is equivalent to the choice of a Weierstrass equation, and the remaining $e_i$ allow one to modify the coordinate in higher degrees.  

We next construct a map $g\colon T=\spec (\mathbb{Z}_{(2)}[\alpha_1, \alpha_3])/\mathbb{G}_m \to
M_{cub}^{coord}/\mathbb{G}_m$ giving a  commutative diagram:
\[ \xymatrix{
 T \ar[rd]^p \ar[r]^g &
 M_{cub}^{coord}/\mathbb{G}_m \ar[d]^{q'} \ar[r] & M_{FG}^{coord}/\mathbb{G}_m
\ar[d]^q  \\
& M_{cub} \ar[r] &  M_{FG}
}.
\]
We define a map $g\colon T \to
M_{cub}^{coord}/\mathbb{G}_m$ by sending the
cubic curve $y^2 + \alpha_1 xy + \alpha_3 y = x^3$ to
the same cubic curve with the canonical coordinate $-x/y$. 
In other words, $a_1 \mapsto \alpha_1, a_3 \mapsto \alpha_3$, and all the other
polynomial generators are mapped to zero. 
This morphism is a closed immersion. 

We let $\mathcal{V}$ be the vector bundle of the previous section: that is,
$\mathcal{V} = p_*(\mathcal{O})$. Then $\mathcal{V}$ is actually a bundle of finite, flat
{commutative algebras} on $M_{cub}$. The diagram naturally furnishes a map 
\[ q'_*(\mathcal{O}) \to q'_*(g_*(\mathcal{O})) \simeq p_*(\mathcal{O}) = \mathcal{V},\] which is a
surjection of sheaves of algebras as $g$ was a closed immersion.
Therefore, we know that 
\( q'_*(\mathcal{O}) \simeq \mathcal{F}(MU),  \)
and the map $\mathcal{F}(DA(1)) \to \mathcal{F}(MU)$ combined with the
surjection $q'_*(\mathcal{O}) \to \mathcal{V}$ induces a map
$\mathcal{F}(DA(1)) \to \mathcal{V}$. This is a morphism of eight-dimensional
vector bundles on $M_{cub}$. 

\subsection{Identification of $\mathcal{F}(DA(1))$}
Our goal is to prove that the map $\mathcal{F}(DA(1)) \to \mathcal{V}$
constructed in the previous section is an isomorphism, or equivalently, that it is a
surjection. The following lemma will be useful. 

\begin{lemma} 
\label{Nakayamavariant}
Let $\mathcal{E} \to \mathcal{F}$ be a morphism of coherent sheaves on the
stack $M_{cub}$ (localized at 2). Let $x\colon \spec (\mathbb{Z}/2) \to M_{cub}$ classify the cuspidal
cubic curve. If $x^* \mathcal{E} \to x^* \mathcal{F}$ is a surjection of
$\mathbb{Z}/2$-vector spaces, then $\mathcal{E} \to \mathcal{F}$ is a
surjection of coherent sheaves. 
\end{lemma} 
\begin{proof}
We can pull back the coherent sheaves to the flat cover $\spec \left(
\mathbb{Z}_{(2)}[a_1, a_2, a_3, a_4, a_6 ]\right)$ to obtain finitely generated
graded modules $E, F$ over 
$\mathbb{Z}_{(2)}[a_1, a_2, a_3, a_4, a_6 ]$ and a map $E \to F$ such that
\[ 
E \otimes_{\mathbb{Z}_{(2)}[a_1, a_2, a_3, a_4, a_6 ]} \mathbb{Z}/2 \to 
F \otimes_{\mathbb{Z}_{(2)}[a_1, a_2, a_3, a_4, a_6 ]} \mathbb{Z}/2
\]
is a surjection. However, Nakayama's lemma, in both its graded and ungraded
forms, now implies that $E \to F$ is a surjection of modules, as desired. 
\end{proof} 

\begin{remark}
\label{gradingx}
We note that if $\mathcal{F}$ is any quasi-coherent sheaf on $M_{cub}$, then 
the $\mathbb{Z}/2$-vector space $x^* \mathcal{F}$ has a canonical grading. 
This was used in 
the proof of Lemma~\ref{Nakayamavariant}, and it also can be interpreted 
as follows: 
 the cuspidal curve $y^2 = x^3$ admits a
$\mathbb{G}_m$-action, and consequently all the vector spaces in question
acquire a canonical $\mathbb{G}_m$-action (i.e., grading).
\end{remark}

We can now prove the main result of this section. 
\begin{proposition} \label{sheafDA1}
The composite map $\mathcal{F}(DA(1)) \to \mathcal{F}(MU) \simeq
q'_*(\mathcal{O}) \to \mathcal{V}$ is an isomorphism of vector bundles on
$M_{cub}$ (in particular, on $\mell$). 
\end{proposition} 

\begin{proof}
Since both bundles are eight-dimensional, it suffices to show that the map is a
surjection. By Lemma~\ref{Nakayamavariant}, it suffices to show that 
if $x\colon \spec (\mathbb{Z}/2) \to M_{cub}$ classifies the cuspidal cubic, then 
\( x^* \mathcal{F}(DA(1)) \to x^*\mathcal{V}  \)
is a surjection of vector spaces.

Now, $DA(1)$ is an \emph{even} 2-local spectrum: in particular, it has free
$\mathbb{Z}_{(2)}$-homology and the AHSS degenerates for $MU_*(DA(1))$. This
implies that
\[ x^* (\mathcal{F}(DA(1))) =  MU_*(DA(1)) \otimes_{MU_*} \mathbb{Z}/2 
 \simeq H_*(DA(1); \mathbb{Z}/2);\]
that is, we can get a description of $\mathcal{F}(DA(1))$ even over the
cuspidal locus in terms of a homology theory. 
The same is true for $x^*(\mathcal{F}(MU)) \simeq H_*(MU; \mathbb{Z}/2)$. 
Note that the grading induced (cf. Remark~\ref{gradingx}) on the vector spaces $x^* \mathcal{F}(DA(1))$
(resp. $x^*\mathcal{F}(MU)$) is simply the grading on homology. 

Let us now describe the two maps
\[ H_*(DA(1); \mathbb{Z}/2) \to H_*(MU; \mathbb{Z}/2), \quad H_*(MU;
\mathbb{Z}/2) \to x^*\mathcal{V} . \]

\begin{itemize}
\item 
First consider the map $H_*(DA(1); \mathbb{Z}/2) \to H_*(MU; \mathbb{Z}/2)$. There is a unique nonzero
indecomposable element $S$ in $H_2(MU; \mathbb{Z}/2)$ and a
nonzero indecomposable element 
$T$ in $H_6(MU; \mathbb{Z}/2)$ with the property that  
\[ \{  1, S, S^2, S^3 , T, ST, S^2 T, S^3 T \}  \]
forms a basis of the image of $H_*(DA(1); \mathbb{Z}/2) \to H_*(MU;
\mathbb{Z}/2)$. 
Under the reduction map $MU \to H \mathbb{Z}/2$ and the induced map $H_*( MU;
\mathbb{Z}/2) \to \A$, $S$ maps to $\xi_1^2$ and $T$ maps to $\zeta_2^2$.
\item 
Consider now the map $H_*(MU; \mathbb{Z}/2) \to x^*(\mathcal{V})$; observe that this is a
morphism of {graded algebras.}
In $x^*(\mathcal{V})$,  one has an element $s$
in degree $2$, an element $t$ in degree $6$, and one has a basis for the algebra
given by \( \{1, s, s^2, s^3, t, st, s^2t , s^3 t\}  \) (cf. the proof of
Proposition~\ref{cover}). 
In fact, as we saw earlier, the algebra is $\mathbb{Z}/2[s, t]/(s^4, t^2)$.
Since $H_*(MU; \mathbb{Z}/2) \to
x^* (\mathcal{V})$ is a surjection of graded algebras, this means that the 
indecomposable element $S$ in degree $2$ of $H_*(MU; \mathbb{Z}/2)$ must map to $s$ in
$x^*(\mathcal{V})$, and the indecomposable element $T$ in degree $6$ must map
to $t$ mod decomposables. 
\end{itemize}

Combining these two observations, it now follows that
the composite map $x^* \mathcal{F}(DA(1)) \to x^*\mathcal{V}$ is an isomorphism, as desired. 
\end{proof}

\subsection{Calculation of $\Tmf_*(DA(1))$}
We now have done the work necessary to compute $\Tmf_*( DA(1))$. 
As before, let $T = \spec( \mathbb{Z}_{(2)}[\alpha_1,\alpha_3])/\mathbb{G}_m$. Note that
it is \emph{not} $T$ and $M_{cub}$ which are relevant to this computation, but
rather $T
\times_{M_{cub}} \mell$ and $\mell$. We will use the descent spectral sequence
\eqref{tmfss} and the determination of $\mathcal{F}(DA(1))$ of the previous
subsections. 

\begin{lemma} \label{isweightedP}
$\mell \times_{M_{cub}} T \simeq \mathbb{P}(1,3)$ is the weighted projective
stack: that is, the complement of the intersection $V(\alpha_1) \cap V(\alpha_3)$ in
$\spec (\mathbb{Z}_{(2)}[\alpha_1, \alpha_3])/\mathbb{G}_m$. 
\end{lemma} 
\begin{proof} 
We recall (cf. \cite{formulaire}) that the substack $\mell \subset M_{cub}$ is
the complement of the closed substack of cuspidal curves cut out by the
vanishing of the modular forms $c_4, \Delta$. It follows that $\mell
\times_{M_{cub}} T$ is the substack of $T \simeq \spec
(\mathbb{Z}_{(2)}[\alpha_1,
\alpha_3])/\mathbb{G}_m$ complementary to that cut out by the vanishing of $c_4,
\Delta$. 
In other words, we need to show that $c_4, \Delta$ generate an ideal
in $\mathbb{Z}_{(2)}[\alpha_1, \alpha_3]$ which contains all elements of sufficiently
large degree. 

We can again show this after reducing mod $2$. Here we use the expressions mod
$2$ for $c_4, \Delta$ of the cubic curve $y^2 + \alpha_1x y + \alpha_3 y =
x^3$ (easily extracted from \cite[p. 57]{formulaire}); they are
given by 
\[ c_4 \equiv \alpha_1^4 , \quad \Delta \equiv (\alpha_1 \alpha_3)^3 +\alpha_3^4.  \]
These together imply that $c_4, \Delta$ cut out the empty subscheme of
$\mathbb{P}^1_{\mathbb{Z}/2}$ and consequently generate a power of the
irrelevant ideal. 
\end{proof} 
\begin{theorem} 
The descent spectral sequence for $\pi_* ( \Tmf \wedge DA(1))$ collapses. The
terms in nonnegative degrees are given (additively) by $\mathbb{Z}_{(2)}[\alpha_1,
\alpha_3]$: that is, $\pi_* \tau_{\geq 0}(\Tmf \wedge DA(1)) \simeq
\mathbb{Z}_{(2)}[\alpha_1,\alpha_3]$. Here $|\alpha_1| = 2, |\alpha_3| = 6$.
\end{theorem} 
\begin{proof} 

We saw in Proposition~\ref{sheafDA1} that if $p\colon \mathbb{P}(1,3) \to \mell$ was the eight-fold cover as
above, then 
\( \mathcal{F}(DA(1)) \simeq p_* ( \mathcal{O}),  \)
so that by the projection formula,
\[ \mathcal{F}(DA(1)) \otimes \omega^j \simeq p_* ( \omega^j).  \]
Here $\omega$ refers to the usual bundle on $M_{FG}$ or $\mell$. 
Over $\mathbb{P}(1, 3)$, it arises simply from a shift in grading, i.e., from
the graded $\mathbb{Z}_{(2)}[\alpha_1, \alpha_3]$-module 
$\mathbb{Z}_{(2)}[\alpha_1, \alpha_3]\iota$
where $|\iota| = -2$. 

Consequently, 
we have the $E_2$-term of the descent spectral sequence:
\begin{equation} \label{DA1ss} H^i( \mell, \mathcal{F}(DA(1)) \otimes \omega^j)  = 
H^i( \mathbb{P}(1, 3), \omega^j).
\end{equation}
 The
cohomology of a weighted projective stack is the same as the classical
cohomology of projective space, but the grading is modified. 
Namely, 
one has to compute the cohomology of $\mathbb{A}^2 \setminus
\left\{(0, 0)\right\} = \spec (\mathbb{Z}_{(2)}[\alpha_1, \alpha_3]) \setminus
V(\alpha_1,\alpha_3)$
and keep track of the grading. 
We find that $H^\bullet(\mathbb{A}^2 \setminus \left\{(0, 0)\right\},
\mathcal{O})$ is the cohomology of the two-term (Cech) complex:
\[ 
\mathbb{Z}_{(2)}[\alpha_1^{\pm 1}, \alpha_3] \oplus \mathbb{Z}_{(2)}[\alpha_1,
\alpha_3^{\pm 1}]
\to \mathbb{Z}_{(2)}[\alpha_1^{\pm 1},\alpha_3^{\pm 1}],
\]
and the cohomology of $\mathbb{P}(1,3)$ is the same, with the grading taken
into account. 
In particular, the spectral sequence \eqref{DA1ss} is concentrated in the
bottom two rows, and each row is easy to describe. 
We have:
\begin{gather*}  H^0( \mathbb{A}^2 \setminus \left\{(0, 0)\right\},
\mathcal{O}) = \mathbb{Z}_{(2)}[\alpha_1, \alpha_3], \\ 
 H^1( \mathbb{A}^2 \setminus \left\{(0, 0)\right\},
\mathcal{O}) = 
\mathbb{Z}_{(2)} \left\{ \alpha_1^{-1}\alpha_3^{-1}, \alpha_1^{-2}\alpha_3^{-1},
\alpha_1^{-1}\alpha_3^{-2}, \dots \right\}.
\end{gather*}

\end{proof} 

By the gap theorem in  Theorem~\ref{blackbox}, we have
\[ \tau_{\geq 0}( \Tmf \wedge DA(1)) \simeq \tmf \wedge DA(1).  \]
In particular, we also get: 
\begin{corollary} 
We have an additive isomorphism
$\tmf_*(DA(1)) \simeq \mathbb{Z}_{(2)}[\alpha_1, \alpha_3]$. 
\end{corollary}

\subsection{Connections with level structures}
\label{levestructureinterpretation}
We start by reviewing the modular interpretation of the eight-fold cover. 
Over the locus $M_{cub}[\Delta^{-1}]$ of (nonsingular) elliptic curves, the
restriction of  the cover $T
\to M_{cub}$ is the forgetful functor from the moduli stack of elliptic curves
with a $\Gamma_1(3)$-structure (i.e., a choice of nonzero 3-torsion point, which here is
$(0,0)$) to $M_{cub}[\Delta^{-1}]$; see \cite{MaRe}.
More generally, over $\mell$, the above cover is the cover $(\mell)_1(3) \to
\mell$ of generalized
elliptic curves with a $\Gamma_1(3)$-structure \cite{DR} (cf. \cite[\S 2]{LN2}). 

The recent work of Hill-Lawson \cite{HillLawson} interprets this cover as a
morphism of \emph{derived} stacks. 
In particular, they construct an even periodic derived version of
$(\mell)_1(3)$ which maps to the derived version of the previously
constructed $\mell$ (away from the prime 3). The global sections of the structure sheaf are denoted
$\Tmf_1(3)$ and its connective cover is $\tmf_1(3)$. 

Let $\otop$ be the sheaf of $E_\infty$-ring spectra on the \'etale site of
$\mell$.
Previously, we 
identified the quasi-coherent sheaf on $\mell$ obtained by taking $\pi_0( \otop
\wedge DA(1))$, and we showed that it identifies with the pushforward of the
structure sheaf along $ (\mell)_1(3) \to
\mell$. As a consequence of
the Hill-Lawson work, we can upgrade this to an identification of
quasi-coherent sheaves on derived stacks.

\begin{corollary} \label{identifyqcsheaves2}
The following two quasi-coherent sheaves of $\otop$-modules (on $\mell$
localized at 2) are
equivalent: 
\begin{enumerate}
\item $DA(1) \wedge \otop$. 
\item The pushforward of the structure sheaf $\otop$ on the derived version of
$(\mell)_1(3)$ to $\mell$.
\end{enumerate}
\end{corollary} 
\begin{proof} 
Both yield quasi-coherent sheaves $\mathcal{Q}_1, \mathcal{Q}_2$ of $\otop$-modules on $\mell$ which are
locally free, and 
Proposition~\ref{sheafDA1}
shows that the vector bundles on $\mell$
given by taking $\pi_0$ are isomorphic. 
In order to produce an equivalence $\mathcal{Q}_1 \simeq \mathcal{Q}_2$, we
need to produce a global section of $\mathcal{Q}_3 = \mathcal{Q}_1 \otimes_{\otop}
\hom_{\otop}(\mathcal{Q}_2, \otop)$. 
Now $\mathcal{Q}_3$ is a sheaf of $\otop$-modules such that on $\pi_0$, one
obtains the vector bundle $\mathcal{V} \otimes \mathcal{V}^{\ast}$ on $\mell$,
and the unit is a global section of $\pi_0 \mathcal{Q}_3$. 
This survives to a global section of $\mathcal{Q}_3$ because $\pi_*
\mathcal{Q}_3$ has only cohomology in $H^0$ and $H^1$ as it is pushed forward
from $(\mell)_1(3)$, which has cohomological dimension one. 
In particular, there are no obstructions to producing the equivalence
$\mathcal{Q}_1 \simeq \mathcal{Q}_2$.
\end{proof} 

Taking global sections and connective covers (using the gap in $\pi_* \Tmf$), one obtains: 
\begin{theorem}[Hopkins-Mahowald]  \label{tmf132}
We have an equivalence of $\tmf_{(2)}$-modules $\tmf_{(2)} \wedge DA(1) \simeq
(\tmf_1(3))_{(2)}$. 
\end{theorem}

\subsection{Analogs at odd primes}
We briefly indicate the modifications in the above arguments that can
be used at an odd prime.

When localized at a prime $p > 3$, the moduli stack $\mell$ can be identified
with the weighted projective stack $\mathbb{P}(4, 6)$: that is, any
(possibly nodal) elliptic
curve over a $\mathbb{Z}[1/6]$-algebra $R$ can be (Zariski locally) written in the form
\( y^2 = x^3 + A x + B,  \)
where $A, B$ do not simultaneously vanish. The isomorphisms between elliptic curves are of
the form $(x,y) \mapsto (u^2 x, u^3 y)$ for $u \in R^*$. 
The elements $A$ and $B$ are, up to units in $R$, the modular forms $c_4, c_6$. 
In particular, the descent spectral sequence for $\Tmf_{(p)}$ runs
\[ H^i( \mathbb{P}(4, 6), \omega^j) \implies \pi_{2j-i} \Tmf_{(p)},  \]
and degenerates, since the cohomology is concentrated in $H^0$ and $H^1$. 
One has therefore
\[ \pi_* \Tmf =
\mathbb{Z}_{(p)}[c_4, c_6] \oplus
\mathbb{Z}_{(p)}\left\{c_4^{-1}c_6^{-1}, c_4^{-2}c_6^{-1}, c_4^{-1}
c_6^{-2}, \dots\right\},  \]
and $\Tmf$ is complex-orientable as it is torsion-free.

Next we consider the prime $3$.
Here there is a three-fold cover  of the moduli
stack of cubic curves. 
We will construct this cover, and show that it can be realized via a three-cell
complex. 

\begin{proposition} 
Let $\overline{p} \colon \spec (\mathbb{Z}_{(3)}[\alpha_2, \alpha_4])/\mathbb{G}_m \to M_{cub}$ classify the  cubic curve
$y^2 = x^3 + \alpha_2 x^2 + \alpha_4 x$. 
Then $\overline{p}$ is a finite flat cover of rank three.
\end{proposition} 

\begin{proof} 
We will imitate the arguments of Proposition~\ref{cover}.
Namely, to show finiteness, we can argue as in the proof of
Proposition~\ref{cover} and
reduce to showing that the pullback $\spec (\mathbb{Z}_{(3)}[a_2, a_4])
\times_{M_{cub}} \spec (\mathbb{Z}/3)$ is the spectrum of a finite
$\mathbb{Z}/3$-algebra (where $\spec (\mathbb{Z}/3) \to M_{cub}$ classifies the
cuspidal cubic). 
Using the change-of-variable formulas, we find that this fiber product is the
spectrum of 
\[ \mathbb{Z}/3[r, s, t]/( 2s, 2t, r^3 - t^2) = \mathbb{Z}/3[r]/(r^3),  \]
which is clearly a finite $\mathbb{Z}/3$-algebra of dimension three. We conclude that for any $R$
and for any map $\spec (R) \to M_{cub}$, the fiber product
$\spec (\mathbb{Z}_{(3)}[a_2, a_4])/\mathbb{G}_m
\times_{M_{cub}} \spec (R)$ is a subscheme of $\mathbb{A}^3_R$ cut out by three
equations, and that it is  finite over $R$. As in the proof of
Proposition~\ref{cover}, this implies
flatness. 
\end{proof} 

In order to realize the vector bundle $\overline{p}_*(\mathcal{O})$ on $\mell$ by a
spectrum, we consider the generating element $\alpha_1 \in \pi_3(S^0)_{(3)} = \mathbb{Z}/3$. The cofiber of $\alpha_1$, which is a desuspension of
$\mathbb{HP}^2$, has cohomology generated by elements $x_0, x_4$ with
$\mathcal{P}^1 x_0 = x_4$. 
Since, furthermore, $\alpha_1^2 =0 \in \pi_6(S^0)_{(3)}  =0$, we conclude that
there is a 3-local finite spectrum $X_3$ such that
\[ H^\bullet(X_3; \mathbb{Z}/3) \simeq \left\{x_0, x_4, x_8\right\}, \quad
\mathcal{P}^1 x_0 = x_4, \ \mathcal{P}^1 x_4 = x_8.  \]
To construct $X_3$, we consider the diagram
\[
\xymatrix{
 & & & S^7 \ar@{-->}[ld]^{\phi} \ar[d]^{\alpha_1} \\
S^3 \ar[r]^{\alpha_1} &  S^0 \ar[r] &  \Sigma^{-4} \mathbb{HP}^2 \ar[r] &  S^4
\ar[r]^{-\Sigma \alpha_1} &  S^1
},
\]
where the horizontal line is a cofiber sequence. 
Since $\alpha_1^2 =0$, we find a lifting $\phi: S^7 \to \Sigma^{-4}
\mathbb{HP}^2$ and let $X_3$ be the cofiber of $\phi$.

In the spirit of the previous sections, we prove:
\begin{proposition}
The vector bundle that $X_3$ defines on $\mell$ is isomorphic to
$\overline{p}_*(\mathcal{O} )$.
\end{proposition}
\begin{proof} 
We follow the outline of the earlier arguments. 
In fact, we start by producing a map $X_3 \to MU$ (implicitly localized at 3) which induces an isomorphism on $\pi_0$, using the degeneration of
the AHSS. In homology, the map
\[ H_*(X_3; \mathbb{Z}/3) \to H_*(MU; \mathbb{Z}/3) \simeq \mathbb{Z}/3[x_1,
x_2, \dots,  ], \quad |x_i| = 2i,  \]
is an embedding whose image contains an indecomposable generator in degree four,
 and its square. As a result, one gets a map of vector bundles on $M_{FG}$,
\( \mathcal{F}(X_3) \to \mathcal{F}(MU),  \)
such that when one takes the fiber over the additive formal group over $\spec
(\mathbb{Z}/3)$, one obtains the above map in homology. 

Next, consider the cover $q\colon  \spec (L)/\mathbb{G}_m \to M_{FG}$ and the
pullback
(cf. Theorem~\ref{mcoordcub})
$$q'\colon   \spec (\mathbb{Z}_{(3)}[a_1, \dots, a_6, \left\{e_n\right\}_{n
\geq 4}])/\mathbb{G}_m \simeq M_{cub}^{coord} \to M_{cub}.$$
As before, $q'_*(\mathcal{O})$ is the sheaf on $M_{cub}$
that one obtains from $MU$. One produces a map 
\( q'_*(\mathcal{O}) \twoheadrightarrow \overline{p}_*(\mathcal{O}),  \)
by considering the closed embedding $$ \spec( \mathbb{Z}_{(3)}[\alpha_2,
\alpha_4])/\mathbb{G}_m \hookrightarrow 
 \spec \left(\mathbb{Z}_{(3)}[a_1, \dots, a_6, \left\{e_n\right\}_{n
\geq 4}]\right)/\mathbb{G}_m,$$
which sends $a_2 \mapsto \alpha_2, a_4 \mapsto \alpha_4$, and annihilates all
the other generators. 

The claim is that the composite map 
\( \mathcal{F}(X_3) \to \mathcal{F}(MU) \to \overline{p}_*(\mathcal{O}),  \)
is an isomorphism of vector bundles on $M_{cub}$, which as before can be checked by showing 
that the map yields a surjection when one takes the fiber over the cuspidal
cubic. To see this, we observe that when one takes the fiber over the cuspidal
cubic, one gets the embedding $H_*(X_3; \mathbb{Z}/3) \to H_*(MU; \mathbb{Z}/3)$, whose image
contains an indecomposable
generator in degree $4$ and its square. The map from $H_*(MU; \mathbb{Z}/3)$ to the fiber of
$\overline{p}_*(\mathcal{O})$ over the cuspidal cubic (which we have checked to
be $\mathbb{Z}/3[r]/(r^3)$) is a map of graded algebras and induces a surjection on
indecomposables. From this, the conclusion follows similarly. 
\end{proof} 

The threefold cover of $\mell$ one obtains here is obtained from
$\Gamma_1(2)$-structures \cite[\S 7]{St12} . 
Using similar techniques as in Theorem~\ref{tmf132} and the Hill-Lawson work
\cite{HillLawson}, one obtains from the above
analysis: 

\begin{theorem} 
One has an equivalence of $\tmf_{(3)}$-modules
$\tmf_{(3)} \wedge X_3 \simeq \tmf_1(2)_{(3)}$.
\end{theorem}

\section{The remaining computations}

In this section, we finish the computations. We use the 2-local complex
$DA(1)$ and the 3-local complex $F$ to analyze the Adams-Novikov spectral
sequence for $\tmf$. 
We then identify the spectrum $\tmf_{(2)} \wedge DA(1) \simeq \tmf_1(3)_{(2)}$ with a form of
$BP\left \langle 2\right\rangle$, 
and using Hopf algebra manipulations, we are able to recover $H_*(\tmf;
\mathbb{Z}/2)$. 

\subsection{Calculation of $\Tmf_*(MU)$ and the stack for $\tmf$}
In this subsection, we 
compute the homotopy groups of $\Tmf \wedge MU$ and $\tmf \wedge MU$.  This
leads to the description of the Adams-Novikov spectral sequence for $\tmf$ in
terms of the moduli stack of cubic curves. 
Throughout, we work {integrally.}

\begin{proposition} 
\label{TmfMU}
Let $R = \mathbb{Z}[a_1, \dots, a_6, \{e_n\}]_{n \geq 4}$.
Then $\Tmf_*(MU)$ is a module over $R$. As an $R$-module, it is
isomorphic to
\( R \oplus C , \)
where $C_* = (R/(c_4^\infty, \Delta^\infty))_{* + 1}$ and $R/(c_4^\infty,
\Delta^\infty)$ is the cokernel of $R[c_4^{-1}] \oplus R[\Delta^{-1}] \to R[(c_4
\Delta)^{-1}]$.
\end{proposition} 
\begin{proof}
This follows from the descent spectral sequence $H^i( \mell, \mathcal{F}(MU)
\otimes \omega^j) \implies \pi_{2j-i} (\Tmf \wedge MU)$.
The sheaf $\mathcal{F}(MU)$ on $\mell$ is obtained by
pushing forward the structure sheaf along the affine map
\[ \mathrm{Spec}(\pi_*MU)/\mathbb{G}_m \times_{M_{FG}} \mell \to \mell.  \]
The source of this morphism is identified with 
$\left( \spec (R) \setminus
 V(c_4, \Delta) \right)/\mathbb{G}_m$ by Theorem~\ref{mcoordcub}, since the locus
 $\mell \subset M_{cub}$ is the complement of the closed substack defined by
 $c_4, \Delta$.

Let $B$ be the scheme 
$\spec (R) \setminus
 V(c_4, \Delta)$. 
 Consequently, we get for the $E_2$ page of the descent spectral sequence for $\Tmf_*(MU)$,
 $$ E_2^{i, 2j}   = H^i( \mell, \mathcal{F}(MU) \otimes \omega^j)
\simeq H^i(B, \mathcal{O})_j,
$$
where the subscript $j$ denotes taking the $j$th piece. 
The scheme $B$ is covered by two affine opens given by the
 localizations at $\Delta$ and $c_4$. In particular, it has cohomological
 dimension one, and for dimensional reasons the descent spectral sequence for
 $\Tmf \wedge MU$ degenerates.
We can describe the cohomology of $B$
as that of the Cech complex
\[ R [c_4^{-1}]
\oplus  R [\Delta^{-1}]
\to 
R[ (c_4 \Delta)^{-1}].
\]
Putting this together, the result follows easily as the $H^0$ gives $R$ and
the $H^1$ gives $R/(c_4^\infty, \Delta^\infty)$. 
\end{proof}

We observe that $C$ can be identified with the ideal of $\Tmf_*(
MU)$ given by elements in positive filtration degree in the filtration from the descent
spectral sequence, and consequently $C^2 = 0$. 

We will now describe the stack corresponding to $\tmf$, or
equivalently describe the structure of the Adams-Novikov spectral sequence for
$\tmf$. 
We will see that the stack associated to $\tmf$
(in the sense of sec.~\ref{sec:stackring})
is precisely the moduli stack
$M_{cub}$ of cubic curves. This produces a spectral sequence that allows
computation with $\tmf$. 

\begin{corollary}[{Cf. \cite[Prop. 20.1]{rezk512}}]
One has $\pi_*(\tmf \wedge MU ) = \mathbb{Z}[a_1, a_2, a_3, a_4, a_6, \{e_n\}]_{n \geq 4}$. \end{corollary} 
\begin{proof} We begin by computing $\tmf_*(MU)$. In fact, we know
(Proposition~\ref{TmfMU}) that $\pi_* ( \Tmf \wedge MU)$  canonically surjects onto the ring
in question with kernel an ideal $C \subset \Tmf_*(MU)$ of square zero, and that there is a map
\( \tmf \wedge MU \to  \Tmf \wedge MU.  \)
The claim is that  the composite
\begin{equation} \label{thismap} \tmf_*(MU ) \to \Tmf_*(MU) \to \Tmf_*(MU)/C
\simeq \Tmf_*(MU)/F_1 \Tmf_*(MU)  \end{equation}
is an isomorphism, where we use the notation of Proposition~\ref{TmfMU} for $C$ and where $F_1$ refers to
the filtration from the descent spectral sequence. 
This will compute $\tmf_*(MU)$, as desired. 
We will prove this locally at each prime $p$. 

In order to prove this, we will use the existence of a finite \emph{even}
$p$-local spectrum $Z$ with the following properties:
\begin{enumerate}
\item 
 $\Tmf \wedge Z$ is a complex-orientable ring spectrum.
\item  $\tmf \wedge Z = \tau_{\geq 0}( \Tmf \wedge
Z)$.
\end{enumerate}

For $p = 2$, we proved this fact with $Z = DA(1)$. 
For $p = 3$, we proved this fact with $Z = X_3$. For $p \geq 5$, we can take $Z =
S^0$, since $\Tmf$ and $\tmf$ are complex-orientable ring spectra with
torsion-free homotopy groups. 

We will now use these three facts to prove that \eqref{thismap} is an isomorphism at
the arbitrary prime $p$.
To start with,  we note that $\tmf \wedge MU$ is a complex-orientable $E_\infty$-ring,
so that, since $Z$ is an even spectrum,
$  \tmf_*( MU \wedge Z )$ is a sum of copies of $\tmf_*(MU)$ (possibly
shifted).
The analog holds for $\Tmf_*( MU \wedge Z)$. 
Moreover, since
the map
\( \tmf \wedge Z \to \Tmf \wedge Z , \)
induces a split injection on homotopy groups, we find that 
\[ \pi_*( \tmf \wedge Z \wedge MU ) \to \pi_*( \Tmf \wedge Z \wedge MU)   \]
is a split injection of graded abelian groups in view of the (known)
$MU$-homology of any complex-oriented ring spectrum.

Since $Z$ is even,
we conclude that $\tmf_*(MU) \to \Tmf_*(MU)$ is a split injection.
It follows that both $\tmf_*(MU)$ and the cokernel of $\tmf_*(MU) \to
\Tmf_*(MU)$ are torsion-free abelian groups. 
To show that $\tmf_* (MU) \cap F_1 \Tmf_*(MU) = 0$ and that 
$\tmf_*(MU) \to \Tmf_*(MU)/F_1 \Tmf_*(MU)$ is an isomorphism, it now suffices to
work \emph{rationally}, since the kernel and the cokernel are torsion-free. This is much easier. 
In fact, rationally, the descent spectral sequence for $\Tmf$ degenerates and is concentrated
in the zeroth and first row. 
The image of $\pi_* \tmf \otimes \mathbb{Q} \to \pi_* \Tmf \otimes \mathbb{Q}$
consists of the elements on the zeroth row: that is, $\pi_* \tmf \otimes
\mathbb{Q} \to \pi_* \Tmf \otimes \mathbb{Q} \to (\pi_* \Tmf/F_1 \pi_* \Tmf)
\otimes \mathbb{Q}$ is an isomorphism. It follows that this is true after
tensoring $\tmf$ with any rational spectrum. 
\end{proof}

In particular, we can describe the spectral sequence for $\tmf$-homology.

\begin{corollary}[Hopkins-Mahowald] 
The stack for $\tmf$ is identified with $M_{cub}$.
In particular, let $X$ be a spectrum. Then $X$ defines a quasi-coherent sheaf $\widetilde{F}(X)$
on the moduli stack $M_{cub}$ (if $X$ is even, then $\widetilde{F}(X)$ is the
pullback of $\mathcal{F}(X)$ under $M_{cub} \to M_{FG}$) and a spectral
sequence
\[ H^i( M_{cub}, \widetilde{F}(X) \otimes \omega^j) \implies \tmf_{2j-i}(X).  \]
\end{corollary} 
In  \cite{rezk512}, Rezk uses instead the Thom spectrum $X(4)$ and reproduces
the Weierstrass Hopf algebroid (cf. \cite[Prop. 12.4 and Th. 14.5]{rezk512}). 
\begin{proof}

This analysis shows that the morphism
\( \tmf_*(MU) \to \Tmf_*(MU) \to \Tmf_*(MU)/C,  \)
is an isomorphism, where $C$ is the ideal consisting of those elements of
filtration $1$ in the (degenerate) descent spectral sequence for $\Tmf \wedge MU$. 
The same is therefore true when $MU$ is replaced by any wedge of suspensions of
$MU$, for instance any smash power of $MU$. 
We find that, for any $s >0$, 
$\tmf_*(MU^{\wedge s})$ can be described as global sections of the structure
sheaf on the stack
\[ \mell^{coord} \times_{\mell} \dots \times_{\mell} \mell^{coord}  \]
(with $s$ factors): that is, the contributions of $H^1$ terms in the
$\Tmf$-homology can be ignored. The ring of global sections over this stack 
is the same as the ring of global sections over the larger stack
\[ M_{cub}^{coord} \times_{M_{cub}} \dots \times_{M_{cub}} M_{cub}^{coord} , \]
which differs by a substack of codimension $2$. We find that 
the Hopf algebroid
\[ \tmf_*(MU) \rightrightarrows \tmf_* (MU \wedge MU) \triplearrows \dots  \]
precisely writes down the presentation of $M_{cub}$ via the flat cover
$M_{cub}^{coord} \to M_{cub}$. 
\end{proof}

The stack $M_{cub}$ can be presented by the Weierstrass Hopf algebroid. We
remark that this spectral sequence is the Adams-Novikov spectral sequence for
$\tmf \wedge X$.
In \cite{computation}, it is used to calculate $\pi_* \tmf$.

\subsection{$\tmf \wedge DA(1)$}
We will now identify the ring spectrum $\tmf_1(3)$  with 
a form of $BP\left \langle 2\right\rangle$. 

We begin with some generalities. 
Implicitly, we work at a prime $p$. 
Given an $E_\infty$-ring spectrum $R$ and an element $r \in \pi_k(R)$, we write
$R/r$ for the cofiber of $r \colon \Sigma^k R \to R$. 
Given a sequence of elements $r_1, r_2, \dots,  \in \pi_*(R)$, we let
$R/(r_1, r_2, \dots )$ denote the colimit over $n$ of the finite smash products $R/r_1
\wedge_R \dots \wedge_R R/r_n$. When $\pi_*(R)$ is concentrated in even degrees
and the $r_i$'s are nonzerodivisors in $\pi_*(R)$, it is known that these can always be
given $A_\infty$-algebra structures in $R$-modules (cf. \cite[Sec.
3]{angeltveit}).

\begin{definition} \label{BPn}
A $MU$-module spectrum $M$ is said to be \emph{a form of $BP\left \langle
n\right\rangle$} (at the prime $p$) if there exist
elements $x_i \in  \pi_{2i}MU$ for $i \in \mathbb{Z}_{>0} \setminus \left\{p-1, p^2 - 1,
\dots, p^n - 1\right\}$ such that one has an equivalence
of $MU$-modules
\[ MU_{(p)}/(\{x_i\} ) \simeq M,  \]
and such that $x_i$ generates $\pi_{2i}(MU)$ modulo decomposables. 
We will frequently abuse notation and write $BP\left \langle n\right\rangle$ for any form
of $BP\left \langle n\right\rangle$. 
\end{definition} 

In general, it does not seem easy to tell whether an $MU$-module spectrum is a
form of $BP\left \langle n\right\rangle$ simply from looking at its homotopy
groups as a module over $MU_* = \pi_*(MU)$. However, the following examples  will be useful in the future. 
We note also that recent work of Angeltveit-Lind \cite{AL15} has shown that the
\emph{spectrum} obtained by $p$-completing $BP\left \langle n\right\rangle$ is determined by its mod $p$
cohomology.

\begin{example} \label{BP2crit}
Suppose $M$ is an $MU$-module spectrum with $$M_* \simeq BP_*/(v_{n+1}, v_{n+2}, \dots,
) \simeq \mathbb{Z}_{(p)}[v_1, v_2, \dots, v_n].$$
Suppose $M$ admits the structure of an \emph{$MU$-ring spectrum}, i.e., the
structure of a possibly nonassociative algebra object in the \emph{homotopy category} of
$MU$-modules, inducing the natural
ring structure on $M_*$. Then $M$ is a form of $BP\left \langle
n\right\rangle$. 

In fact, consider the map $MU_* \to M_*$ induced by the unit $MU \to M$, and choose indecomposable generators
$y_i$ in $\pi_{2i}(MU)$ (for each $i \neq p - 1, p^2 -1 ,\dots, p^n-1$) for the
kernel. The unit map $MU \to M$ extends over $MU/y_i \to M$, and we can take
the smash product of all of these together and localize at $p$ to get a map
$MU_{(p)}/(\left\{y_i\right\}) \simeq BP\left \langle n\right\rangle \to M$,
which is an isomorphism on homotopy groups. 
\end{example}

Observe that $\tmf \wedge MU$ is an $E_\infty$-ring spectrum, so that we can
make various quotients in the category of $\tmf \wedge MU$-modules. Consider
in particular 
\( (\tmf _{(2)}\wedge MU)/(a_2, a_4, a_6, \{e_n\}).   \)
The next lemma will enable us to analyze it. 

\begin{lemma}\label{MUmodulething} The homotopy groups of the $MU$-ring
spectrum $\tmf_{(2)} \wedge MU/( a_2, a_4, a_6,\{
e_n\})$ are given by $\mathbb{Z}_{(2)}[a_1, a_3]$, and it is a form of $BP\left \langle 2\right\rangle$. 
\end{lemma} 

Let $I $ denote the ideal $ (a_2, a_4, a_6, \{e_n\}) \subset R$; we will (by abuse of notation) write
$(\tmf_{(2)} \wedge MU)/I$ for the quotient $(\tmf_{(2)} \wedge MU)/(a_2, a_4, a_6,
\left\{e_n\right\})$. 

\begin{proof} 
It suffices to describe the Hurewicz map $MU_* \to M_*$, in view of
Example~\ref{BP2crit}. By \cite[Cor. 3.2]{angeltveit}, $M$ is an $A_\infty$-algebra in
$MU$-modules. 

Given an elliptic spectrum $E$ associated to an elliptic curve over $\spec
(R)$, the map $MU_* \to E_*(MU)$ yields
the map $\spec (E_*(MU)) \to \spec  (MU_*)$ from the ring
classifying coordinates on the formal group 
 to the ring classifying formal
group laws. In particular, the map 
\[ MU_* \to \tmf_*(MU) \simeq \Tmf_*(MU)/F_1 \Tmf_*(MU)  \]
classifies the formal group law constructed by choosing the coordinate $-y/x +
\sum_{n \geq 4} e_n (-y/x)^{n+1}$ on the
formal group of $y^2 + a_1 xy + a_3 y = x^3 + a_2 x^2 + a_4 x + a_6$. 
The composite $MU_* \to (\Tmf_*(MU)/F_1 \Tmf_*(MU))/I$  classifies the formal
group law associated to the coordinate $-y/x$ on $y^2 + a_1 xy + a_3 y = x^3$. 

Here we use the formulas given in  \cite[Ex. 4.5, Ch. 4]{silverman}. 
The expansion of the power series $[2](z)$ for the formal group law associated to
a Weierstrass equation is  
\[ [2](z) =  2 z - a_1 z^2 - 2a_2 z^3 + (a_1 a_2 - 7a_3 )z^4  + \dots,  \]
where in our case $a_2 = 0$ and $7$ is invertible. 
In particular, if we take $v_1, v_2$ to be  indecomposable elements of $MU$ in
degrees $2$ and $6$ (e.g., the coefficients of $z^2, z^4$ in $[2](z)$), and
take our equation to be $y^2 + a_1 xy + a_3 y = x^3$, then $v_1$ maps to a unit times $a_1$, and $v_2$ maps to a
unit times $a_3$. It follows that there exist indecomposable generators
in $\pi_i(MU)$ for $i \neq 2, 6$ which generate the kernel of the surjective map
$\pi_*(MU)
\to \mathbb{Z}_{(2)}[a_1,a_3]$. This proves the result. 
\end{proof} 
\begin{corollary} \label{Tmfmodules}
The map of $\tmf$-modules
\( \tmf \wedge DA(1) \to (\tmf \wedge MU)/I\)
is an equivalence.
\end{corollary} 
\begin{proof} 
This follows from 
Proposition~\ref{sheafDA1} and from the construction of the vector bundle $\mathcal{V}$
used there. 
Note that $\pi_*(\tmf \wedge DA(1))$ is isomorphic to the quotient of $\pi_*(
\Tmf \wedge DA(1))$ by the elements in filtration $1$. 
\end{proof}

Combining Lemma~\ref{MUmodulething} and Corollary~\ref{Tmfmodules}, we find:
\begin{theorem}[Hopkins-Mahowald \cite{HM}] 
\label{tmfwedgeDA1}
$\tmf_{(2)} \wedge DA(1)  \simeq \tmf_1(3)_{(2)} \simeq  \tau_{\geq
0}(\Tmf_{(2)} \wedge DA(1))$ is a form of $BP \left \langle
2\right\rangle$. 
\end{theorem} 

\subsection{The mod 2 homology of $\tmf$}
In the previous subsection, we established the 2-local equivalence $\tmf \wedge
DA(1) \simeq BP \left \langle 2\right\rangle$. Using Hopf algebra manipulations and
the homology of $BP \left \langle 2\right\rangle$, we can now calculate the
homology of $\tmf$.

We begin 
by recalling the homology of $BP\left \langle n\right\rangle$. 
The result is classical and follows from the spectral sequence
$\mathrm{Tor}_{\pi_* MU}(\pi_* BP\left \langle n\right\rangle; H_*(MU;
\mathbb{Z}/2)) \implies H_*(BP\left \langle
n\right\rangle; \mathbb{Z}/2)$; see also \cite[Th. 4.4]{LN2}.
\begin{proposition} \label{BPnhomology}
The map $BP \left \langle n-1\right\rangle \to H \mathbb{Z}/2$ induces an
injection on homology, with image 
\( \mathbb{Z}/2[ \zeta_1^2, \zeta_2^2, \dots, \zeta_{n}^2, \zeta_{n+1},
\zeta_{n+2}, \dots 
] \subset \A .  \)
\end{proposition}

We will identify the subalgebra of $H_*( BP \left \langle
2\right\rangle; \mathbb{Z}/2)$ given by the homology of $\tmf$. 
First, we need a few lemmas. 
Let $k$ be a field. 
As usual, a nonnegatively graded $k$-algebra $R$ is called \emph{connected} if $k \to R_0$ is an
isomorphism, and a graded $R$-module  $M$ is called \emph{connective} if it is zero in
negative degrees. Given $R$, we let $\overline{R} = \bigoplus_{i >0} R_i$.
We recall the following: 

\begin{lemma}[Nakayama's lemma] \label{NAK}
Let $R$ be a graded, connected $k$-algebra with augmentation ideal
$\overline{R}$. 
If $M$ is a connective $R$-module which is flat, then it
is free (in particular, if in addition $M \neq 0$, then $M$ is faithfully flat). 
\end{lemma} 

\begin{lemma} \label{sublemhopf}
Let $B$ be a commutative, graded connected Hopf algebra over a field $k$ and
let $A \subset B$ be a graded comodule subalgebra. 
Then we can write the inclusion $A \subset  B$ as a filtered
colimit of inclusions
\( A_j  \subset B_j , \ j \in J, \)
where $B_j \subset B$ is a \emph{finitely generated} commutative  graded,
connected Hopf
subalgebra, and $A_j   \subset B_j$ are \emph{finitely generated}
subalgebras
which are also  comodules for $B_j$. 
\end{lemma}
By ``commutative graded,'' we do not mean ``graded-commutative.''
 The lemma is false without the graded and connected hypotheses: for instance,
consider the ring of functions on the multiplicative group, $k[t, t^{-1}]$ (a
Hopf algebra with $\Delta(t) = t \otimes t$), and
the comodule subalgebra $k[t]$. Geometrically, this corresponds to the
$\mathbb{G}_m$-variety $\mathbb{A}^1$: the natural inclusion map $\mathbb{G}_m
\to \mathbb{A}^1$ (of $\mathbb{G}_m$-varieties)
is not faithfully flat. 
We remark that the existence of such structure theorems for comodule algebras in the graded,
connected case goes back to \cite{MM}.

\begin{proof} 
First, observe that $B$ is a filtered colimit of finitely generated graded Hopf
subalgebras \cite[Lemma 21.1.2]{concise}. 
Observe also that $A$, as a $B$-comodule, is a filtered colimit of
finite-dimensional graded $B$-comodules. Given a finite-dimensional $B$-comodule
$M \subset A$, it is, as a result, a $\widetilde{B}$-comodule for
$\widetilde{B}$ a  large enough finitely generated (graded) subHopf algebra of $B$. Now,
consider the pairs $(\widetilde{A}, \widetilde{B})$ for $\widetilde{A}$ the
subalgebra of $A$ generated by $M$. Then $\widetilde{A}$ is a finitely
generated algebra and a $\widetilde{B}$-comodule. 
Clearly, given any finite collection of elements of $A$ and any finite
collection of elements of $B$, we can find a pair $(\widetilde{A},
\widetilde{B})$ containing all of them. 
\end{proof} 
\begin{proposition} \label{freecomodule}
Let $B$ be a commutative graded, connected Hopf algebra over a field $k$,
 and let $A \subset A' \subset B$ 
be comodule subalgebras. Suppose $B$ is a domain. Then $A'$ is graded free over $A$. 
\end{proposition} 
\begin{proof} 
We will show that $B$ is faithfully flat over $A$. This also implies that $B$
is faithfully flat over $A'$, and combining these observations shows that $A'$
is faithfully flat over $A$. Lemma~\ref{NAK} will then imply that $A'$ is a free
$A$-module. In other words, we may assume $A' = B$.

In the notation of Lemma~\ref{sublemhopf}, we will prove that each of the inclusions $A_j \subset B_j$ is faithfully  flat, which
will suffice; in fact we only need to check flatness by Lemma~\ref{NAK} again. 
Thus, we may assume $A$ and $B$ are finitely generated. By making a base
change, we may also assume that $k$ is \emph{algebraically closed.} Then $\spec
B$ is an affine group scheme $G$, of finite type over $k$, and $X = \spec A$ is
a scheme acted on by $G$.
There is a map of $G$-schemes
\( p\colon  G \to X  \)
which is dominant (since $A \subset B$). 
Since $X$ is an integral scheme, generic flatness (see, e.g.,
\cite[Theorem 14.4]{eisenbud}) implies that there exists a
nonempty
open subset $U \subset X$ such that  $p^{-1}(U) \to U \to X$ is flat. Thus for any $g
\in G(k)$, we have that
\( gp^{-1}(U) \to X   \)
is flat, and since $G$ is the union of the $g p^{-1}(U)$ for $ g \in G(k)$,  we find
that $G \to X$ is flat. 
\end{proof}

Using the above technical tools, we can now deduce our main result. 

\begin{theorem}[Hopkins-Mahowald \cite{HM}] 
The map $\tmf \to H \mathbb{Z}/2$ induces an injection on mod 2 homology, and
we have an identification
\[ H_*(\tmf; \mathbb{Z}/2) = \mathbb{Z}/2[ \zeta_1^8,
\zeta_2^4, \zeta_3^2,
\zeta_4, \zeta_5, \dots ] \subset \A.  \]
\end{theorem} 
\begin{proof} 
In fact, we know that 
\[ H_*( BP \left \langle 2\right\rangle; \mathbb{Z}/2) \simeq \mathbb{Z}/2[
\zeta_1^2, \zeta_2^2, \zeta_3^2,
\zeta_4, \zeta_5, \dots ],  \]
by Proposition~\ref{BPnhomology}; the map comes from the truncation $BP \left \langle
2\right\rangle \to H \mathbb{Z}_{(2)} \to H \mathbb{Z}/2$ which embeds
$H_*(BP\left \langle 2\right\rangle; \mathbb{Z}/2)$ as a subcomodule of $\A$. 
We also know that $H_*(\tmf; \mathbb{Z}/2)$ is a comodule algebra, and the factorization
\[ \tmf \to \tmf \wedge DA(1) \simeq BP \left \langle 2\right\rangle \to H
\mathbb{Z}/2  \]
shows that it is a subcomodule algebra of $H_*(BP\left \langle 2\right\rangle;
\mathbb{Z}/2) \subset \A$. Moreover, the K\"unneth formula shows that the graded
dimension of $H_*(\tmf; \mathbb{Z}/2)$ is that of $\mathbb{Z}/2[\zeta_1^8,
\zeta_2^4,
\zeta_3^2, \zeta_4, \dots ]$.

We will show that if $\mathcal{C}$ is any  subcomodule
algebra of $H_*(BP \left \langle 2\right\rangle; \mathbb{Z}/2)$ with the same
graded dimension as $\mathbb{Z}/2[\zeta_1^8,
\zeta_2^4,
\zeta_3^2, \zeta_4, \dots ]$, then $\mathcal{C}$ is in fact $\mathbb{Z}/2[\zeta_1^8,
\zeta_2^4,
\zeta_3^2, \zeta_4, \dots ]$ (which is easily checked to be a
valid subcomodule algebra). Here we will show that any two $\mathcal{C},
\mathcal{C}'$ satisfying that condition are equal. 
In fact, if $\mathcal{C}, \mathcal{C}'$ satisfy the condition, then
Proposition~\ref{freecomodule}
shows that $H_*(BP\left \langle 2\right\rangle; \mathbb{Z}/2)$ is free (necessarily of rank eight) over
each of $\mathcal{C}, \mathcal{C}'$. Consider the subcomodule algebra
$\mathcal{C}'' \subset H_*(BP\left \langle 2\right\rangle; \mathbb{Z}/2)$
generated by $\mathcal{C}, \mathcal{C}'$.
It also has the property that $H_*(BP;
\mathbb{Z}/2)$ is free over $\mathcal{C}''$, by Proposition~\ref{freecomodule} again. 
By counting the rank, we will arrive at a contradiction. We need first:

\begin{lemma} 
The  only elements of $\mathcal{C}''$ in degrees $ < 12$ are $1$ and $\xi_1^8 =
\zeta_1^8$.
\end{lemma} 
\begin{proof} 
In fact, we know that
\( \dim \mathcal{C}_8 = \dim \mathcal{C}'_8 = 1,  \)
and since this is the smallest dimension of a nonzero element in each of
$\mathcal{C}, \mathcal{C}'$, the generating element must be primitive. The
primitive elements in the dual Steenrod algebra $\A$, considered as a comodule
over itself, are $1$
and $\{\xi_1^{2^n}\}_{n
\geq 0}$, though, so $\mathcal{C}_8 =
\mathcal{C}'_8$ is generated by $\xi_1^{8}$. There are no other elements in
degrees $<12$ in $\mathcal{C}$ or $\mathcal{C}'$. 
\end{proof} 
We claim that the eight elements $1, \xi_1^2, \xi_1^4, \xi_1^6,
\zeta_2^2,
\xi_1^2 \zeta_2^2, \xi_1^4 \zeta_2^2, \xi_1^6
\zeta_2^2  \in H_*(BP\left \langle
2\right\rangle; \mathbb{Z}/2)$ are linearly independent in $H_*(BP \left
\langle 2\right\rangle; \mathbb{Z}/2) \otimes_{\mathcal{C}''} \mathbb{Z}/2$
(i.e., could be taken as a subset of generators over $\mathcal{C}''$).
This is a consequence of the fact that the only elements in $\mathcal{C}''$ of
degree less than $12$ are $1, \xi_1^8$. 
Moreover, in  degree $12$, we observe that $\xi_1^6 \zeta_2^2
\notin \mathcal{C}''$ as it is not primitive modulo $\xi_1^8$. 
Consequently, the rank of $H_*(BP\left \langle 2\right\rangle ;
\mathbb{Z}/2)$ as a $\mathcal{C}''$-module must be at least
eight.  
This means that the graded dimension of $\mathcal{C}''$ must be equal to that
of $\mathcal{C}$, so $\mathcal{C} = \mathcal{C}' = \mathcal{C}''$.
\end{proof} 

The dual assertion describes the cohomology via
\[ H^*( \tmf; \mathbb{Z}/2) \simeq \mathcal{A} \otimes_{\mathcal{A}(2)}
\mathbb{Z}/2,  \]
where $\mathcal{A}(2) \subset \mathcal{A}$ is the subalgebra of the Steenrod
algebra $\mathcal{A}$ generated by $\mathrm{Sq}^1, \mathrm{Sq}^2,
\mathrm{Sq}^4$. In fact, $H^*(\tmf; \mathbb{Z}/2)$ is cyclic over
$\mathcal{A}$ since $H_*(\tmf; \mathbb{Z}/2) \subset \A$. Dimensional
restrictions force $\mathrm{Sq}^1, \mathrm{Sq}^2$, and $\mathrm{Sq}^4$ to
annihilate the generator in degree zero, and this produces a surjection
\[ \mathcal{A} \otimes_{\mathcal{A}(2)} \mathbb{Z}/2 \twoheadrightarrow H^*(\tmf;
\mathbb{Z}/2).  \]
Since $\mathcal{A}(2) \subset \mathcal{A}$ is a \emph{Hopf} subalgebra, 
$\mathcal{A}$ is free over $\mathcal{A}(2)$ by the results of \cite{MM}, and we
find that the graded dimensions of $\mathcal{A} \otimes_{\mathcal{A}(2)}
\mathbb{Z}/2 $ and $H^*(\tmf; \mathbb{Z}/2)$ match. This proves the asserted
description of the cohomology. 

By the change-of-rings theorem $\mathrm{Ext}^{s,t}_{\mathcal{A}}( \mathcal{A}
\otimes_{\mathcal{A}(2)} \mathbb{Z}/2, \mathbb{Z}/2) \simeq
\mathrm{Ext}^{s,t}_{
\mathcal{A}(2)}(\mathbb{Z}/2, \mathbb{Z}/2)$, we now conclude: 
\begin{corollary} 
The (mod 2) Adams spectral sequence for $\tmf$ runs
\[ \mathrm{Ext}_{\mathcal{A}(2)}^{s,t}( \mathbb{Z}/2, \mathbb{Z}/2) \implies
\pi_{t-s} \tmf \otimes \mathbb{Z}_{2}.  \]
\end{corollary} 

The Adams spectral sequence for $\tmf$ is displayed in \cite[Ch. 13]{TMF}. 
We remark that (as is well-known) it is not possible to continue this process, and realize modules
of the form $\mathcal{A} \otimes_{\mathcal{A}(n)} \mathbb{Z}/2$ for $n \geq 3$,
where $\mathcal{A}(n) \subset \mathcal{A}$ is generated by $\left\{\mathrm{Sq}^1, \dots,
\mathrm{Sq}^{2^n}\right\}$, because of the solution to the Hopf invariant one
problem.  

This method of computing the homology (and Adams-Novikov spectral sequence) of $\tmf$ 
is dependent on a key piece of prior computational knowledge about $\Tmf$:
namely, the gap theorem in the homotopy groups $\pi_* \Tmf$. It would be
interesting if one could give a theoretical explanation of the gap theorem. 
\bibliographystyle{alpha}
\bibliography{tmf}
\end{document}